\documentclass[final]{amsart}
\usepackage[T1]{fontenc}
\usepackage[usenames]{color}
\usepackage{float,amsmath}
\usepackage{eucal}
\usepackage{graphicx}
\usepackage{multicol}
\usepackage{hyperref}
\newlength{\mcm}
\newlength{\imcm}
\newlength{\jmcm}
\newtheorem{prop}{Proposition}

\newtheorem{Thm}{Theorem}

\newtheorem{rem}{Remark}
\newtheorem{remark}{Remark}
\newtheorem{Remark}{Remark}

\newcommand{\ms}[1]{\mathbb{#1}}
\newcommand{\mc}[1]{\mathcal{#1}}
\newcommand{\lip}{Lipschitz }

\newcommand{\R}{\mathbb{R}}
\newcommand{\ve}{\varepsilon}

\def\1{\mathchoice {\rm 1\mskip-4mu l} {\rm 1\mskip-4mu l}
{\rm 1\mskip-4.5mu l} {\rm 1\mskip-5mu l}}
\newcommand{\Id}{\mathrm{Id}}
\newcommand{\mQ}{m}

\begin{document}
%\addtolength{\baselineskip}{+1.\baselineskip}
\author{Alain Trouv\'e and Fran\c{c}ois-Xavier Vialard.}
\address[Alain Trouv\'e]{CMLA \\Ecole Normale
  Sup\'erieure de Cachan, CNRS, UniverSud\\ 61, avenue du Pr\'esident Wilson\\ F-94 235
  Cachan CEDEX}
\email{alain.trouve@cmla.ens-cachan.fr}
\address[Fran\c{c}ois-Xavier Vialard]{Institute for Mathematical Science\\ Imperial College London\\
53 Prince's Gate, SW7 2PG, London, UK}
\email{francois.xavier.vialard@normalesup.org}
\title[Shapes splines]{Shape Splines and Stochastic Shape Evolutions: \\A
Second Order Point of View}
\date{\today}
\maketitle
\tableofcontents
\begin{abstract}
This article presents a new mathematical framework to perform
statistical analysis on time-indexed sequences of 2D or 3D shapes. 
At the core of this statistical analysis is the task of time
interpolation of such data. Current models in use 
can be compared to linear interpolation for one dimensional data. 
We develop a spline interpolation method which is directly 
related to cubic splines on a Riemannian manifold. Our strategy 
consists of introducing a control variable on the Hamiltonian
equations of the geodesics. Motivated by statistical modeling 
of spatiotemporal data, we also design a stochastic model to deal 
with random shape evolutions. This model is closely related to the 
spline model since the control variable previously introduced is set 
as a random force perturbing the evolution.

Although we focus on the finite dimensional case of
landmarks, our models can be extended to infinite dimensional shape
spaces, and they provide a first step for a non parametric growth model
for shapes taking advantage of the widely developed framework of 
large deformations by diffeomorphisms.
\end{abstract}
\section{Introduction}
Mathematical and statistical modeling of shapes has undergone
significant development over the past twenty years
driven by a wide range of applications in medical imaging. Initially
the focus was on the comparison between two
shapes also refereed to as registration. Among others, registration and comparison tools 
derived from a Riemannian point of view on shapes spaces and diffeomorphic transport
have been actively developed during the past few years.
% This comparison has been made with
% the help of Riemannian metrics on the space of shapes and the research
% has been actively developed since the past few years. 
This  framework
was used to represent shapes and study the statistical variation of
static shapes within a population. An emerging question of interest is
now to study time dependent data of shapes (images, landmarks,
surfaces or tensors). The basic dataset is then a sequence of shapes
indexed by time. For example, a practical application would be the
analysis of follow-up studies in brain imaging.

Several attempts models for the variability of
longitudinal data have been proposed recently: a parametric model 
of growth is proposed in
\cite{DBLP:conf/isbi/GrenanderSS06,DBLP:journals/tmi/GrenanderSS07},
which aims to describe the biological evolution as an iteration of
random elementary diffeomorphisms, so called GRID. Focusing on image
data, statistical estimation of the parameters are performed with the
GRID model in
\cite{DBLP:conf/iciar/PortmanGV09,DBLP:conf/emmcvpr/SrivastavaSDG05}. Other
attempts are often based on using an initial registration tool (such as
geodesics on a group of diffeomorphisms, see \cite{Younes2009S40} for
a large overview) to interpolate the time dependent data with
piecewise geodesics \cite{durlemann,4541298}. In \cite{durlemann}, the
model is further developed with the introduction of time realignments to
allow the study of an ensemble of longitudinal data and the computation 
of an averaged space-time evolution.

From the modeling point of view a typical growth evolution is
usually smooth in time. However, the piecewise geodesic
models underlying current analysis of time dependent shape data 
can not prevent a loss of regularity at the observation
points. Moreover, from a more classical statistical point of view,
piecewise linear regression does not provide the best interpolation
framework and  from a probabilistic point of view the limiting process 
underlying piecewise geodesic interpolation is more or less a Kunita
flow \cite{Kunita} that has a Brownian like time evolution. 
These remarks suggest that a better model of interpolation in time 
should be of higher order than piecewise geodesics, and also closely 
related to a probabilistic model of random evolution of shapes. 
The question of time interpolation was
addressed in \cite{4408977} with the use of a kernel on the time
variable but still with underlying piecewise geodesics on the
data. Random evolutions of shapes have been treated for instance in
\cite{drydenbook,kendalldiffusion}, but here again the model is of
first order in the sense that the evolution is not smooth in time.

In this work, we present a second order model for deterministic and 
stochastic shape evolution as a primary step toward further
statistical applications. This work will extensively use the 
Hamiltonian framework that emerged several years ago in the field of 
Computational Anatomy (CA) to compute the registration of landmarks 
(points of interest on an image)
\cite{gty06,landmark_matching,michor-2007-23}. This problem of
landmark registration is equivalent to finding geodesics on the
Riemannian manifold of landmarks. Thus shooting methods as well as
gradient descent have been applied to solve this problem
\cite{bb37524,1122513}.

A first natural step to deal with random shape evolution is 
to study a stochastic perturbation of the Hamiltonian equations of 
geodesics. If we consider the evolution of landmarks as a physical 
system of particles, it corresponds to the introduction of a random 
force on their evolution providing smooth time random perturbation 
around a mean geodesic trajectory. However, this in turn leads to a 
new deterministic counterpart when the stochastic term in the
evolution of the momentum is replaced by a deterministic 
control variable. Given a sequence of time-indexed shapes, and
optimizing on the control variable, we end up with a framework
for smooth time interpolation between shapes. 
% It gives rise to mathematical questions such as the 
% possible blow-up of the stochastic
% solutions.
% Once this question is solved, the question of a mean
% trajectory of a given sample of shape evolutions leads us to extend
% this simple model to describe non-geodesic evolutions of shapes. To
% this aim we need a deterministic framework to interpolate time
% sequences of shapes. A natural idea is to replace the stochastic term
% in the evolution of the momentum by a control variable. Hence we would
% end up with a smooth time interpolation between shapes.
In the finite dimensional case of landmarks, this approach is directly 
related to splines on a Riemannian manifold
\cite{Crouch,Noakes1,jackson90:_dynam_inter_and_applic_to_fligh_contr,splinesCk}.
We will present a Hamiltonian approach to derive the equations for 
the splines in a Hamiltonian setting. It can be related to the work of 
\cite{parallel} since they use a control point of view to obtain the 
Hamiltonian equations but our approach differs from theirs in that we use the
Hamiltonian formulation of geodesics on landmark space. Hence 
there is a straightforward generalization to other Hamiltonian systems.
The possible extension of the Hamiltonian equations of geodesics to
curves, surfaces (or embeddings of manifolds in $\ms{R}^n$) is the key
feature to design second order models of evolution on continuous shape spaces
\cite{gty06,michor-2007-23,VialardThesis} but in this work we will
concentrate on the finite dimensional case of landmarks.
Note however, that contrary to the initial context of air-craft
trajectory planning (where the target manifold is the low-dimensional Lie Group 
$\mathrm{SE}(3)$ - the Special Euclidean group on $\mathbb{R}^3$), for
which splines on Riemannian manifold were introduced - shape spaces in 
Computational Anatomy are more strongly connected with the infinite
dimensional case. Moreover, we will consider both the deterministic
and the stochastic setting since these are considered as fundamental
ingredients towards statistical analysis of longitudinal data.

The plan of the paper is as follows: we first introduce in section
\ref{HamSec} the Hamiltonian equations for the case of landmarks, which
constitute the central tool in this article. The section \ref{SplineModel} is
devoted to the spline model with different metrics and section
\ref{sec:Existence} gathers theoretical results about the global
existence in time of the second order deterministic spline evolution, 
the existence of a minimizer for the spline estimation problem and
the derivation of the Euler-Lagrange equations. Then in section
\ref{sec:numexp}, we present some numerical experiments about shape 
spline estimation on 2D shape evolution, highlighting some of their
main features. In section \ref{StoModel}, we turn to problem of 
stochastic second order spline evolutions with a proof of the
well-posedness and global existence in time and a few illustrative
simulations. Last we develop in section \ref{sec:homogeneous} a
slightly more general picture of shape splines on homogeneous spaces 
extending the model beyond the landmark setting.
\section{Hamiltonian equations of geodesics for landmark matching}\label{HamSec}

The problem of landmark matching via diffeomorphic transport is now quite well understood. The basic idea is to build a continuous path of minimal length $\phi_t$ starting from $\phi_0=\Id_{\mathbb{R}^d}$  in a group $G_V$ of 
diffeomorphisms, between an initial configuration
$x=(x_i)_{1\leq i\leq n}$ of $n$ landmarks in $\mathbb{R}^d$ and a target configuration $y=(y_i)_{1\leq i\leq n}$~. Thus, if $x_{i,t}\doteq\phi_t(x_i)$ and $x_t=(x_{i,t})_{1\leq i\leq n}$, $x_t$ is a path from $x$ to $y$ in the space of landmark configurations induced by a geodesic in a Riemannian space of diffeomorphisms. The group $G_V$ is defined through the flow of time 
dependent velocity fields $(t,x)\to v_t(x)$ on $\mathbb{R}^d$
\begin{equation}
  \label{eq:15-1}
  \frac{\partial \Phi}{\partial t}=v_t\circ \Phi
\end{equation}
where $V$ is a Hilbert space of velocity fields and $v\in
L^2([0,1],V)$ is an element of  the space of time dependent velocity fields with finite $L^2$
norm. Obviously, the existence of a flow $t\to \Phi^v_t$ for $v\in
L^2([0,1],V)$ solution of (\ref{eq:15-1}) depends on some regularity
assumptions of the instantaneous velocity fields $v_t$, mainly the control
of the first order derivatives of $v_t$. Assuming this, the group $G_V$
is defined as $G_V\doteq \{\Phi^v_1\ |\ v\in L^2([0,1],V)\}$ and the
diffeomorphic matching problem for landmarks is formulated through the
following variational problem

\begin{equation}
  \label{eq:15-2}
\left\{
  \begin{array}[h]{l}
    \min \int_0^1|v_t|_V^2dt,\ v\in L^2([0,1],V)\\
\text{with}\\
\Phi^v_1(x_i)=y_i,\ 1\leq i\leq n\,.
  \end{array}\right.
\end{equation}

The problem (\ref{eq:15-2}) is well posed as soon as $V$ is
continuously embedded in $C^1_0(\mathbb{R}^d,\mathbb{R}^d)$ the space
of $C^1$ velocities vanishing at $\infty$ (such a $V$ is called an
admissible space). Namely, we assume the existence of a constant $C$ 
such that for any $v \in V$, the following inequality is verified
\begin{equation} \label{C1inj}
|v|_{1,\infty} \leq C |v|_V \, . 
\end{equation}
Such an admissible space is a Reproducible Kernel
Hilbert Space (RHKS) and is equipped with a kernel $K_V(z,z')$ playing a
key role in defining the structure of the solution of the geodesic
emerging from (\ref{eq:15-2}). Since the reader may not be familiar
with RKHS, let us say in a nutshell that for any $z,z'\in
\mathbb{R}^d$ the kernel $K_V(z,z')$ is a $d\times d$ matrix in
$\mathcal{M}_d(\mathbb{R})$, that $z\to K_V(z,z')\alpha'\in V$ for any
$z',\alpha'\in \mathbb{R}^d$, that $\langle
K_V(.,z')\alpha',K_V(.,z'')\alpha''\rangle_V=\alpha'^TK_V(z',z'')\alpha''$
(the so called reproducible property) and that $\langle v,
K_V(.,z')\alpha'\rangle_V=\langle v(z'),\alpha'\rangle_V$ for any $v\in
V$. The main fact is that given $K_V$, the space $V$ is completely
defined and one may start from the choice of the kernel $K_V$
itself to define the space $V$. A large number of 
kernels have been proposed most commonly the Gaussian kernel
\begin{equation}
K_V(z,z')=\exp(-\frac{|z-z'|^2}{\lambda^2})\Id_{\mathbb{R}^d}\,.\label{eq:15-4}
\end{equation}

Even more importantly, the kernel appears explicitly in 
geodesics emerging from (\ref{eq:15-2}) as described by the following Theorem.
To ensure the existence of a solution in this theorem, we need to assume that the kernel is positive definite:
%% YA Don't you need P.D to define the RKHS????
\begin{equation}\label{controlabilite}
|\sum_{i=1}^l K_V(x_i,.)\alpha_i|_V^2 = 0 \Rightarrow \forall \; i \; \alpha_i = 0 \,.
\end{equation}

\begin{Thm}[cf \cite{gty06}]The solution  $v\in L^2([0,1],V)$ exists and satisfies
  \begin{equation}
v_t(z)=\sum_{i=1}^n K_V(z,x_{i,t})p_{i,t}\label{eq:15-5}
\end{equation}
where $t\to (x_t,p_t)$ is solution of 
\begin{equation}
  \label{eq:15-3}
\left\{
  \begin{array}[h]{l}
      \dot{x}_t=\frac{\partial H_0}{\partial p}(x_t,p_t)\\
\\
\dot{p}_t=-\frac{\partial H_0}{\partial x}(x_t,p_t)
  \end{array}
\right.
\end{equation}
with 
$H_0(p,x)\doteq \frac{1}{2}\sum_{i,j}p_i^TK_V(x_i,x_j)p_j$
and $x_0=x$. 
\end{Thm}
An interesting fact is that the evolution (\ref{eq:15-3}) is a 
Hamiltonian evolution induced by the \textbf{Hamiltonian} $H_0$ on the system of 
landmarks $x=(x_1,\cdots,x_n)$. From a mechanical point of view, the 
landmarks represent the positions of $n$ particles in 
$\mathbb{R}^d$ and each $p_i$ is the \textbf{momentum} attached to the
particle $x_i$. During the evolution, the particles  interact
and the equation $\dot{p_i}=-\frac{\partial H_0}{\partial x_i}(p,x)
$ implies that the time derivative of the momentum of the $i^{th}$ particle is equal
to the \textbf{internal forces} $-\frac{\partial H_0}{\partial x_i}(p,x)$ 
acting on it. This is a simple extension of the usual 
\textbf{Newton equations}. 

\begin{remark}
Note that when there is no interaction between the particles i.e. $K_V(z,z')=\frac{\mathbf{1}_{z=z'}}{m}\text{Id}_{\mathbb{R}^d}$, then $H_0(p,x)=\frac{1}{2m}\sum_{i=1}^n |p_i|^2$ and the system (\ref{eq:15-3}) reduces to $m\dot{x}_i=p_i$ and $\dot{p}_i=m\ddot{x}_i=0$ so that the particles evolve independently along straight lines at constant speed. Though this kernel is not admissible, it is not even continuous, this case can be view as a limit case when the characteristic scale $\lambda\to 0$ in (\ref{eq:15-4}). When the particles interact, 
the evolutions of the particles are no longer straight lines. The velocity of a given particle is the aggregate of the contributions coming from every particle as described in (\ref{eq:15-5}).
\end{remark}

\vspace{0.2cm}

This Hamiltonian formulation can be efficiently retrieved by the application of the Pontryagin Maximum Principle (PMP)\cite{MR2062547}.
The previous assumption in \eqref{controlabilite} implies that the following control problem is controllable (since for every $\mathbf{v}=(v_1,\ldots,v_n)  \in \mathbb{R}^{dn}$
there exists $u \in V$ such that $\mathbf{v}= u\cdot x\doteq
(u(x_i))_{1\leq i\leq n}$).
% \begin{equation}
% \begin{cases}
% \dot{x} = u \cdot x \, \\
% x(i) = x_i \;\text{ for } i=0,1
% \end{cases}
% \end{equation}
For any chosen (initial and final) configurations of landmarks, there exists a path between the two landmarks sets. 
Following an optimal control point of view, we can introduce the Hamiltonian of the system
\begin{equation}
H(p,x,u) = \langle p,\dot{x} \rangle - \frac{1}{2}\langle u,u \rangle_V = \langle p,u \cdot x \rangle - \frac{1}{2}\langle u,u \rangle_V 
\end{equation}
minimizing $H$ w.r.t. to $u$ we obtain:
\begin{equation}
u(.) = k(.,x)p \,,
\end{equation}
with the compact notation $k(.,x)p \doteq \sum_{i=1}^n K_V(.,x_i)p_i$. Then the minimized Hamiltonian can be written as:
\begin{equation}
H_0(p,x) = \frac{1}{2} \sum_{i=1}^n \langle p_j , K_V(x_j,x_i) p_i
\rangle  = \frac{1}{2} \langle x ,k(q,q) x\rangle \,. 
 \label{eq22.2.1}
\end{equation}
Hence we get the Hamiltonian equations of the previous theorem by directly applying the PMP.

\vspace{0.2cm}

It is rather standard to prove that this approach endows the landmark space (set of groups of distinct points) with a Riemannian metric which is induced from that of the group of diffeomorphisms. More generally once a positive definite kernel is given, it gives rise to a Riemannian metric on the space of landmarks. We then  use the term kernel metric to designate such a Riemannian metric on the space of landmarks.
% \\
% The question of the completeness of the riemannian manifold depends on the kernel. For example, the Gaussian kernel gives rise to a complete metric on the contrary to the exponential kernel. 
% \\

The Hamiltonian framework (also viewed from an optimal control viewpoint) can be generalized to shape spaces where the landmark space $\mathcal{L}$ is replaced with with $L^2(\ms{R}/\ms{Z},\ms{R}^2)$ for closed curves in the plane \cite{gty06} or measures. It has been also generalized to discontinuous images \cite{VialardThesis}. However when deforming embedded objects in $\ms{R}^d$ through the action of smooth vector fields,  exact matching usually impossible. The associated control problem is not controlable any more. This is the reason why the Hamiltonian equations are established for the inexact matching problem which includes a penalty term in the minimization. The effect of this penalty term is often to smooth the structure of the momentum used in the Hamiltonian formulation.
\section{Spline interpolation on landmark space: a generative growth model}\label{SplineModel}
\subsection{The standard growth model}
The usual growth model formulation as derived in \cite{mty02} was
initially defined in the case of images. Assume that one has a 
continuous time sequence of noisy image data $I_t^D,\ t\in [0,1]$, and 
an image template $I_0$. The basic growth estimation problem is 
to estimate a path  $\phi_t,\ t\in [0,1]$ in the diffeomorphic group $G_V$
(a sequence of deformations) such that $I^D_t\simeq
\phi_t\cdot I_0$ with $\phi\cdot I=I\circ \phi^{-1}$. From a Bayesian
perspective, this inverse problem is cast as the minimization
 of the cost
\begin{equation}
   J^I(v)=\frac{1}{2}\int_0^1|v_t|_V^2dt +
  \frac{\lambda}{2}\int_0^1|I^D_t-\phi_t^v \cdot I_0|^2dt
\end{equation}
where $\phi^v_t$ is constrained to be the flow of $v$ starting at 
$\Id$ at time $0$ and $v\in L^2([0,1],V)$. Note that the underlying 
stochastic model for $v$ is white noise in $L^2([0,1],V)$ so that
the associated flow $\phi$ is a  \textbf{Kunita flow} \cite{Kunita}
with almost nowhere differentiable trajectories $t\to\phi_t$.

In the landmark setting addressed in this paper, 
assume that we have a sequence $x^D_t$ of data (where
$x^D_t=(x^D_{t,i})$ is an $n$-tuple of $d$ dimensional landmarks) and a
template $x_0$. A straightforward adaptation of the above growth model
leads to the new cost to be minimized
\begin{equation}
  J^x(v)=\frac{1}{2}\int_0^1|v_t|_V^2dt +
  \frac{\lambda}{2}\int_0^1|x^D_t-\phi_t^v\cdot x_0|^2dt
\label{eq:21.12.1}
\end{equation}
where $\phi\cdot x_0 =(\phi(x_{0,i}))$, and the squared distance $|x^D_t-\phi_t^v\cdot x_0|^2$ comes from a standard Gaussian white noise model the on individual landmarks. 
As above the underlying stochastic growth model is a Kunita flow for $\phi_t$ with irregular trajectories (see Fig. \ref{Kunita}). 

However, considering that $x^D$ is continuously observed in time, the minimization problem (\ref{eq:21.12.1})  can be cast into an optimal control problem  associated to
the reduced Hamiltonian $H$ given by \cite{macki1982introduction}
$$H(p,x,t)=H_0(p,x)-\frac{\lambda}{2}|x^D_t-x|^2
\text{ where }H_0(p,x)=\frac{1}{2}\sum_{i,j}p_i^TK(x_i,x_j)p_j\,.$$

The associated optimal trajectories are solutions of the Hamiltonian
flow:
\begin{equation}
  \label{eq:21-10}
\left\{
  \begin{array}[h]{l}
      \dot{x}_t=\frac{\partial H_0}{\partial p}(p_t,x_t)\\
\\
\dot{p}_t=-\frac{\partial H_0}{\partial x}(p_t,x_t)+\lambda (x_t-x^D_t)
  \end{array}
\right.
\end{equation}
The optimal 
 time dependent vector field $v$ can be
reconstructed through the equality
$$v_t(z)=\sum_{i} K(z,x_{t,i})p_{t,i}\,.$$

When $\lambda\to 0$, the optimal evolution converges to a geodesic
evolution given by the usual Hamiltonian $H_0$. This is not surprising
since in \eqref{eq:21-10} the second term of the right hand side
vanishes and leads to the minimization of the kinetic energy.

Note that the filtered trajectory $x$, obtained from the observation
 $x^D$, is $C^1$ in time for continuous time observations. It thus has the
advantage of correcting the poor prior stochastic growth model given by Kunita flows. 
However, this is due to the fact that we are assuming
 \emph{continuous} time observations. In contrast, for discrete  observation times $t_k$, the optimal 
solution is piecewise geodesic, again no more $C^1$, and offers
limited interpolation properties. A better prior on $x$ is needed.

\subsection{Perturbative growth model}\label{stochmodel} 
Denoting $u_t=\lambda(x_t-x^D_t)$ the perturbation from the geodesic evolution,  the estimated trajectory $x$ 
extracted from the noisy observation $x^D_t,\ t\in [0,1]$ can be 
reconstructed from the integration of 
\begin{equation}
  \label{eq:25.1}
\left\{
  \begin{array}[h]{l}
      \dot{x}_t=\frac{\partial H_0}{\partial p}(p_t,x_t)\\
\\
\dot{p}_t=-\frac{\partial H_0}{\partial x}(p_t,x_t)+u_t
  \end{array}
\right.
\end{equation}
derived from the ODE (\ref{eq:21-10}) if 
one knows the initial momentum $p_0$ and $u$. The fundamental idea is that 
this (ODE) can play the role of a generative engine for $C^1$ trajectories
if we put the proper constraint on $u$ seen as a \emph{control} variable.
Interestingly, from a mechanical perspective, $u_t$ can be interpreted as\textbf{ external forces} acting on the landmark configuration ($u_{t,i}$ is the force acting on particle $i$). In absence of external force, the landmark configuration follows a geodesic evolution. When the external forces do not 
vanishing, the trajectory followed by these landmarks deviates from a 
simple geodesic to various interpolation trajectories. The question of the dynamics of these external forces is of importance.
The first step is to note that when $u_t=\lambda(x-x^D_t)$, the cost of the filtered trajectory $x$ can be written as
\begin{equation}
  \label{eq:23.1}
  J^x(v)=\frac{1}{2}\int_0^1|v_t|_V^2dt +\frac{1}{2\lambda}\int_0^1|u_t|^2dt\,,
\end{equation}
so the minimization of $J^x$ provides a control on both terms of the right hand side,
both positive quantities.
Moreover, for a given distribution $P_{x^D}$ on $x^D$, the associated 
distribution $P_x$ on the filtered trajectories $x$ can be defined 
through the generative model (\ref{eq:21-10}) for the adequate 
distribution $P_0$ on $(p_0,u)$. Very little can be said on the distribution 
$P_0$, but as a first approximation we can proceed as follows
If the overall model minimization of (\ref{eq:23.1}) for any observation 
$x^D$ provides a reasonable filtered trajectory $x$ and since for such $x$, 
the quantity $\frac{1}{2\lambda}\int_0^1|u_t|^2dt\leq J^x(v)$ should stay 
small (at least controlled), one can put \textbf{a constraint} on its 
expectation  under $P_0$ 
\begin{equation}
  \label{eq:24.1}
  E_0\doteq E_{P_0}(\int_0^1|u_t|^2dt)\,.
\end{equation}
Using no more information and maximizing entropy \cite{edwin1957jaynes,kapur1989maximum} under the constraint (\ref{eq:24.1}), a reasonable first order model for the marginal distribution of $u$ 
is given by  a \textbf{standard white noise} i.e. $u_t=\sigma dB_t$ where it is reasonable to set $\sigma^2$  close\footnote{The different approximations done in this derivation should encourage us to be a little cautious about any strong statement!}  to the value $\lambda$.

The marginal distribution on $p_0$ is more problematic, as well as the conditional distribution of $u$ given $p_0$. The simplest solution is to assume independence of $p_0$ and $u$:
$$P_0(dp_0,du)=P_0(dp_0)\otimes P_0(du)$$
so that that conditional generative model ($p_0$ fixed) for the trajectory $x$ can be 
defined as the $SDE$
\begin{equation}
  \label{eq:24.2}
\left\{
  \begin{array}[h]{l}
      d{x}_t=\frac{\partial H_0}{\partial p}(p_t,x_t)dt\\
\\
d{p}_t=-\frac{\partial H_0}{\partial x}(p_t,x_t)dt+\sigma dB_t
  \end{array}
\right.
\end{equation}
\subsection{New growth estimation minimization problem}
This new growth model  generates  $C^1$ solutions and can be 
the foundation for a new formulation of the growth estimation problem in 
the realistic situation of sparse discrete observation times. Indeed, if 
we consider $M$ observation times $t_1,\cdots,t_M$, this new 
prior can be used to derive a new growth estimation method~:
\begin{equation}
  \label{eq:25.2}
\left|
  \begin{array}[h]{l}
    \displaystyle{ \inf_{p_0,u}J(p_0,u)\doteq \left\{E(p_0)+\frac{1}{2}\int_0^1|u_t|^2dt + \gamma \sum_{k=1}^M |x^D_{t_k}-x_{t_k}|^2\right\}}\\
\\
\text{subject to $(x,p)$ solution of the ODE (\ref{eq:25.1})  }
  \end{array}\right.
\end{equation}
where $(p_0,u)\in \mathbb{R}^{nd}\times L^2([0,1],\mathbb{R}^{nd})$ with initial conditions $(x_0,p_0)$ and $E$ comes from the log-likelihood of the 
prior $P_0(dp_0)$ on $p_0$.

\begin{Remark}   
 The choice of the regularization term $E$ for $p_0$ does not seem important in the finite dimensional case we are considering here.
An improper flat prior can be used (as we do in the application below) and the term can be dropped.
Alternatively consider $p_0$ fixed to $0$. The weight on $u$ is however of particular
importance since it controls the deviation from a geodesic evolution corresponding to $u\equiv 0$.
\end{Remark}
  One can consider more general penalties on $u$ such as
$$\int_0^1 \langle K_{x_t}u_t,u_t\rangle dt$$
for a state dependent metric ($K_x$ is assumed here to be a positive
symmetric matrix). However, there should be some
rationale for the final choice.

At this point we should mention  the Riemannian cubic spline point of view 
developed in \cite{Noakes1}. In this framework, we start from a finite dimensional Riemannian manifold $M$ and the basic problem is to interpolate between two configurations $(x_0,\dot{x}_0)$ and $(x_1,\dot{x}_1)$ with a smooth curve $\gamma$ minimizing an extended \emph{bending} energy~:
$$
\left|
\begin{array}[h]{l}
  \min_\gamma B(\gamma)\doteq \int_0^1 |\nabla_{\dot{\gamma}}\dot{\gamma}|_\gamma^2dt\\
\text{subject to}\\
 \gamma(0)=x_0,\ \dot{\gamma}(0)=\dot{x}_0,\ \gamma(1)=x_1\text{ and }\dot{\gamma}(1)=\dot{x}_1
\end{array}\right.$$
where $\nabla$ is the Levi-Civita connexion and $|\ |_\gamma$ is the 
metric given by the manifold at the current location $\gamma$. It is quite clear that in the situation $M=\mathbb{R}^d$ with the flat Euclidean metric, we get $B(\gamma)=\int_0^1|\ddot{\gamma}|^2dt$ corresponding to the 
classical cubic splines \cite{schoenberg46:_contr_to_probl_of_approx,ahlberg1967theory}. Moreover, one can check (see appendix) that in our Hamiltonian framework, we have 
\begin{equation}
  \label{eq:31.1}
  u=\dot{p}+\partial_xH_0=K_x^{-1}\nabla_{\dot{x}}\dot{x}
\end{equation}
where $K_x\doteq (K_V(x_i,x_j))_{ij}$ and $K_x^{-1}$ is the metric tensor 
at the current landmark positions. Thus, the bending energy is given in our situation as a function of $u$~:
$$B(u)=\int_0^1\langle K_x u,u\rangle dt=\int_0^1|u|_{x,*}^2dt$$
where $|u|_{x,*}$ is the induced metric on the cotangent space.
\begin{Remark}
  This 
penalization is different from the simpler $L^2$ norm we derived 
previously. On one hand, the 
Riemannian cubic splines are completely defined by a unique metric 
underlying the Riemannian structure of the manifold, keeping a pure 
geometric and intrinsic point of view. However, this intrinsic point
of view is not as natural as it may look at first glance since it links 
tightly the internal forces (given by the term $\partial_x H_0$) and 
the external forces (given by $u$) to the same metric structure. This 
is quite questionable since the internal (resp. the external) forces 
proceed from intrinsic (resp. extrinsic) phenomenons.
\end{Remark}

We advocated in the previous subsection to \emph{link the choice of
the metric on the control $u$ to the noise model on the observation data}. 
We derived the particular case of the
white noise situation, which produces the standard Euclidean flat metric on the control. However, more general situations could be of some interest leading to non-standard metrics on
$u$.

\subsection{A non standard metric}
Assume that there exist a Hilbert space $W$ and a smooth mapping $\psi:\mathbb{R}^{nd}\to W^*$
such that the data term in (\ref{eq:23.1}) is replaced by
$$\frac{\lambda}{2}\int_0^1 |\psi(x^D_t)-\psi(x_t)|_{W^*}^2dt\,,$$
where $|\ |_{W^*}$ is the dual norm on the dual space $W^*$.
This situation is quite natural if we consider that the true observed quantity
is not $x^D$ but an element $\psi(x^D)\in W^*$. This setting is of
particular importance in the situation of shape modelling where the
individual label of the different landmarks cannot be observed and the
point-wise correspondence between the template and the observation is problematic. The so-called measure framework developed in \cite{glaunes2004diffeomorphic} makes intensive use
 of such a data term where $\psi(x)=\frac{1}{n}\sum_{i=1}^n
 \delta_{x_i}$ is the 
empirical distribution of the landmarks and $W$ is a proper
Reproducible Kernel 
Hilbert Space.

In this new case, the associated optimal trajectories are solutions of the Hamiltonian flow~:
\begin{equation}
  \label{eq:26.1}
\left\{
  \begin{array}[h]{l}
      \dot{x}_t=\frac{\partial H_0}{\partial p}(p_t,x_t)\\
\\
\dot{p}_t=-\frac{\partial H_0}{\partial x}(p_t,x_t)+\lambda (\psi'(x_t))^\dagger(\psi(x^D_
t)-\psi(x_t))
  \end{array}
\right.
\end{equation}
where $\psi'(x_t)^\dagger:W^*\to\mathbb{R}^{nd}$ is the Hilbertian
adjoint of the differential $\psi'(x_t):\mathbb{R}^{nd}\to W^*$ of
$\psi$ at $x$. As above, denoting $u_t=\lambda
(\psi'(x_t))^\dagger(\psi(x^D_t)-\psi(x_t))$ and assuming that
$\psi'(x_t)^\dagger\psi'(x_t):\mathbb{R}^{nd}\to \mathbb{R}^{nd}$ is
invertible (i.e. $\psi'(x_t)$ is one to one
) along the trajectory $x_t$, we get
\begin{eqnarray*}
  \label{eq:25.2b}
  \frac{\lambda}{2}\int_0^1 |\psi(x^D_t)-\psi(x_t)|_{W^*}^2dt &\geq&
  \frac{\lambda}{2}\int_0^1
  |\pi_{x_t}(\psi(x^D_t)-\psi(x_t))|_{W^*}^2dt\\ &=&
\frac{1}{2\lambda}\int_0^1\langle[\psi'(x_t)^\dagger\psi'(x_t)]^{-1}u_t,u_t\rangle dt\,,
\end{eqnarray*}
where $\pi_x$ is the orthogonal projection in $W^*$ on
$\text{Im}(\psi'(x))$ which is the tangent space at $\psi(x)$ of the
sub-manifold of $W^*$ defined by the immersion $\psi$.
In particular, we can use as local metric
$K_x=[\psi'(x)^\dagger\psi'(x)]^{-1}$ which is now position
dependent and defines a new Riemannian metric on the Landmark space.

In the mentioned situation of measure representation of landmarks we have the following proposition~:
\begin{prop}\label{prop:10.3.1}
  Assume that the kernel $K_W$ associated with the RKHS $W$ is $C^2$. Then $\psi:\mathbb{R}^{nd}\to W^*$ defined by $\psi(x)=\frac{1}{n}\sum_{i=1}^n \delta_{x_i}$ is differentiable and $[\psi'(x)^\dagger\psi'(x)]$ is the block diagonal matrix~:
$$[\psi'(x)^\dagger\psi'(x)]=\frac{1}{n^2}(\frac{\partial^2K_W}{\partial x_{i}\partial x_{j}}(x_i,x_j))
_{1\leq i,j\leq n}\,.$$
\end{prop}
\begin{proof}
  See appendix.
\end{proof}
\section{Existence results and Euler-Lagrange equation for shape
  spline estimation}
\label{sec:Existence}
In this section we provide several rigorous results concerning the shape
spline framework. Indeed, although the overall picture is formally quite
clear, it is necessary to provide rigorous statements about the main
objects we have just introduced. In particular, conditions for the
existence of solutions, global in time, of the perturbed evolution 
(\ref{eq:25.1}) need to be specified, as well as an existence theorem
for the growth estimation minimization problem~(\ref{eq:25.2}). On a more 
practical side, the minimization problem will be solved by gradient descent,
for which we will compute the directional derivatives 
of functional $J$ and the associated Euler-Lagrange equations.

\subsection{Existence of controlled evolution and weak dependency in  the
control variable}
 First define the function 
$f(q,u)\doteq (\partial_pH_0(x,p),-\partial_xH_0(x,p)+u)^T$ for $q=(x,p)\in
\mathbb{R}^{nd}\times\mathbb{R}^{nd}$ and $u\in
\mathbb{R}^{nd}$. We will consider the 
following hypothesis on $V$ and its kernel $K_V$:
\begin{description}
\item[$\mathrm{H0}$ ] \emph{ $V$ is continuously embedded in
   $C^1_0(\mathbb{R}^d,\mathbb{R}^d)$ and its kernel $K_V$ is
   $C^2$ in each of its variable.}
\end{description}
Following is the first result on the existence in time of a solution
to the perturbed evolution (\ref{eq:25.1}).
\begin{prop}[Existence of solutions, global in time, of the controlled system] Assume ($\mathrm{H0}$). Then for any $u\in
  L^2([0,T],\mathbb{R}^{nd})$ and for any initial condition
  $q_0=(x_0,p_0)\in\mathbb{R}^{nd}\times\mathbb{R}^{nd}$, there exists
  $q=(x,p)\in C([0,T],\mathbb{R}^{nd}\times\mathbb{R}^{nd})$ such that 
  \begin{equation}
    \label{eq:23.2.1}
\left\{
    \begin{array}[h]{l}
      x_t=x_0+\int_0^t \partial_pH_0(x_s,p_s)ds\\
      p_t=p_0+\int_0^t(-\partial_xH_0(x_s,p_s)+u_s)ds\,.
    \end{array}\right.
  \end{equation}
Moreover, there exists $C_1,C_2,C_3>0$ independent of $u$ and $q_0$ such that
for any $t\leq T$ we have
\begin{enumerate}
\item $H_0(x_t,p_t)\leq \tilde{H}_T\doteq C_1(H_0(x_0,p_0)+T\int_0^T |u_s|^2ds)$,
\item $|x_t|\leq |x_0|+C_2T\tilde{H}_T^{1/2}$,
\item $|p_t|\leq (|p_0|+\int_0^T|u_s|ds)\exp(C_3T\tilde{H}_T^{1/2})$.
\end{enumerate}
\label{Thm:23.2.1}
\end{prop}
\begin{proof}
Under ($\mathrm{H0}$), we have existence and uniqueness
locally in time of the solution to the system (\ref{eq:23.2.1}) for any
initial condition $q_0=(x_0,p_0)$. The only point to be checked is
that the solution does not go to infinity in finite time. Let $t_0<T$ be such
that we have a solution $q_t=(x_t,p_t)$ defined on $[0,t_0[$. From
(\ref{eq:23.2.1}), we get that
$H_t\doteq H_0(x_t,p_t)=H_0(x_0,p_0)+\int_0^t\langle \partial_pH_0(x_s,p_s),u_s\rangle
ds$. 
From (\ref{C1inj}), since $\partial_{p_i}H_0(x_s,p_s)=v(x_i)$ for
$v(.)=\sum_{j=1}^n K_V(.,x_j)p_j$, we get $|\partial_pH_0(x_s,p_s)|\leq C
H_0(x_s,p_s)^{1/2}$ for a universal constant $C$ and we deduce that
$H_t\leq H_0(x_0,p_0)+C(\max_{s\in [0,t]}H_s)^{1/2}\int_0^{t}|u_s|ds$. Hence 
\begin{equation}
\max_{s\in [0,t]}H_s\leq 2H_0(x_0,p_0)+4C^2(\int_0^t|u_s|ds)^2\leq
2H_0(x_0,p_0)+4C^2T\int_0^T|u_s|^2ds\leq \tilde{H}_T
\label{23.2.4}
\end{equation}
for $C_1=\max(2,4C^2)$. The upper bound does not depend on $t_0$ and 
in particular the Hamiltonian can not explode in finite time. It is
sufficient now to prove that $x_t$ and $p_t$ stay also bounded. Indeed, we
have 
\begin{equation}
|x_t|\leq |x_0|+|\int_0^t\partial_pH_0(x_s,p_s)ds|\leq |x_0|+C\int_0^T
H_s^{1/2}ds\leq |x_0|+CT\tilde{H}_T^{1/2}\,.\label{23.2.5}
\end{equation}
Moreover, $|p_t|\leq \int_0^t |\partial_xH_0(x_s,p_s)|_\infty
|p_s|ds+ \int_0^t|u_s|ds+|p_0|$ and since again $|\partial_{x_i}H_0(x,p)|=|dv^*(x_i)p_i
|$ for $v(.)=\sum_{j=1}^n K_V(.,x_j)p_j$, we get from (\ref{C1inj})
that there exists $C'>0$ such that $|p_t|\leq
C'\int_0^tH_s^{1/2}|p_s|ds+ \int_0^t|u_s|ds+|p_0|$. Using Gronwall's
Lemma, we get eventually
\begin{equation}
 |p_t|\leq (|p_0|+\int_0^T|u_s|ds)\exp(C'\int_0^tH_s^{1/2}ds)\leq (|p_0|+\int_0^T|u_s|ds)\exp(C'T\tilde{H}_T^{1/2})\,.\label{23.2.6}
\end{equation}
From (\ref{23.2.4}), (\ref{23.2.5}) and (\ref{23.2.6}) we deduce that
$x_t$ and $p_t$
stay uniformly bounded on $[0,t_0[$. Since $t_0<T$ is arbitrary, the
system (\ref{eq:23.2.1}) admits a solution on $[0,T]$ satisfying
points 1), 2) and 3) of Proposition \ref{Thm:23.2.1}.
\end{proof}
\begin{prop}[Dependence in $u$]
  \label{prop:24.2.1}
 Assume ($\mathrm{H0}$). For any $q_0\in\mathbb{R}^{nd}\times \mathbb{R}^{nd}$
and any $u\in  L^2([0,T],\mathbb{R}^{nd})$, let $q^{u,q_0}\in
    C([0,T],\mathbb{R}^{nd}\times \mathbb{R}^{nd})$ denote the
    solution of (\ref{eq:23.2.1}) with initial condition $q_0$. Then
    the mapping $(q_0,u)\to q^{u,q_0}$ is continuous for the weak
    topology on $L^2([0,T],\mathbb{R}^{nd})$ and uniform
    convergence on $C([0,T],\mathbb{R}^{nd}\times \mathbb{R}^{nd})$.
\end{prop}
\begin{proof}
Let $u^n\rightharpoonup u^\infty$ be a weakly converging sequence in
$L^2$ and $q^n_0\to q^\infty_0$ be a converging sequence of initial conditions. There exists $R>0$ such that $\sup_{n\geq 0}|u^n|_2\leq R$ and
by Proposition \ref{Thm:23.2.1}, if $q^n$ (resp. $q^\infty$) denotes
the solution of (\ref{eq:23.2.1}) for $(q_0,u)=(q^n_0,u^n)$ (resp. $(q_0,u)=(q^\infty_0,u^\infty)$), 
then there exists $M>0$ such that $|q^n_t|\leq M$ for $(t,n)\in [0,T]\times
\mathbb{N}\cup\{\infty\}$. Let $K_M>0$ such that 
$$|f(q,u)-f(q',u')|\leq K_M|q-q'|+|u-u'|$$
for $|q|,|q'|\leq M$ and $u,u'\in \mathbb{R}^{nd}$ (such $K_M$ exists
since ($\mathrm{H0}$) implies that $dH_0$ is $C^1$). We have 
\begin{eqnarray*}
|q^n_t-q^\infty_t| & =
&|\int_0^tf(q^n_s,u^n_s)-f(q^\infty_s,u^\infty_s)ds+(q^n_0-q^\infty_0)|\\
&\leq&
K_M\int_0^t|q^n_s-q^\infty_s|ds+|\int_0^t(u^n_s-u^\infty_s)ds|+|q^n_0-q^\infty_0|\label{23.2.16}
\end{eqnarray*}
so that using Gronwall's Lemma we get 
\begin{equation}
  \label{23.2.15}
  |q^n_t-q^\infty_t|\leq (|q^n_0-q^\infty_0|+\sup_{r\leq T}|\int_0^r(u^n_s-u^\infty_s)ds|)\exp(K_MT)\,.
\end{equation}
Since $u\to \int_0^ru_sds$ is a continuous linear mapping, we get from
the weak convergence that $|\int_0^r(u^n_s-u^\infty_s)ds|\to 0$ as
$n\to\infty$. Noticing that $|\int_r^{r'}(u^n_s-u^\infty_s)ds|\leq
2\sqrt{r'-r}R$ for any $r<r'\leq T$, Ascoli's Theorem turns the previous
simple convergence into  uniform convergence in $r$.
\end{proof}
\begin{Thm}[Existence of a minimizer in the inexact case]
  Assume ($\mathrm{H0}$) and that $q\to{\mathcal K}_q\in{\mathcal M}_d(\mathbb{R})$
  is $C^0$ with $\langle{\mathcal K_q}u,u\rangle\geq c|u|^2$ for some
  fixed $c>0$ and that $C(q,u)=\frac{1}{2}\langle \mathcal{K}_qu,u\rangle$. 
  Assume that $q_0\to E(q_0)$ is a non negative and
  lower semi-continuous function such that $E\to +\infty$ when $|q_0|\to\infty$, that $g_k$ is a continuous non
  negative function for any $1\leq k\leq M$ and let
  $x_0\in\mathbb{R}^{nd}\times\mathbb{R}^{nd}$ be a fixed initial
  condition. Then the function 
$$J(q_0,u)\doteq E(q_0)+\int_0^T C(q^{q_0,u}_t,u_{t})dt
+\sum_{k=1}^Mg_k(q_{t_k}^{q_0,u})$$ 
defined for any $(q_0,u)\in \mathbb{R}^{nd}\times L^2([0,T],\mathbb{R}^{nd})$ (where $q^{q_0,u}$ is the
solution of (\ref{eq:23.2.1}) with initial condition $q_0$)
reaches its minimum.
\label{Thm:24.2.2}
\end{Thm}
\begin{proof}
Let $(q^n_0,u^n)_{n\geq 0}$ be a minimizing sequence for $J$. Since the $g_k$'s are non
negative and $E\geq 0$ with $E\to +\infty$ when $|q_0|\to +\infty$, we
get that $\int_0^T \langle{\mathcal
  K}_{q^n_s}u^n_s,u^n_s\rangle ds$ and $|q^n_0|$ are upper bounded. Since we assume
that $\langle{\mathcal K_q}u,u\rangle\geq c|u|^2$, we deduce that
$u^n$ is a bounded sequence in $L^2$ and we get by weak compactness
of the strong balls that (up to the extraction of a sub-sequence) there 
exists $u^\infty$ such that $u^n\rightharpoonup u^\infty$. Again, up
to the extraction of a subsequence, we can assume that there exists
$q^\infty_0$ such that $q^n_0\to q^\infty_0$. Since
\begin{eqnarray*}
\int_0^T\langle {\mathcal K}_{q^n_s}u^n_s,u^n_s\rangle ds &\geq& 
\int_0^T\langle {\mathcal K}_{q^n_s}-{\mathcal K}_{q^\infty_s})u^n_s,u^n_s\rangle ds+
\int_0^T\langle {\mathcal K}_{q^\infty_s}u^\infty_s,u^\infty_s\rangle ds\\
&+&2\int_0^T\langle {\mathcal K}_{q^\infty_s}u^\infty_s,u^n_s-u^\infty_s\rangle ds\,,
\end{eqnarray*}
we deduce from Proposition \ref{prop:24.2.1} that 
$$\sup_{s\leq T}\|{\mathcal K}_{q^n_s}-{\mathcal K}_{q^\infty_s}\|\to
0\text{ and }
\int_0^T\langle ({\mathcal K}_{q^n_s}-{\mathcal K}_{q^\infty_s})u^n_s
,u^n_s\rangle ds\to 0\,.$$ Moreover, by weak convergence 
$\int_0^T\langle {\mathcal K}_{q^\infty_s}u^\infty_s,u^n_s-u^\infty_s\rangle ds\to 0$ 
so that 
$\int_0^T C(q^\infty_s,u^\infty_s)ds\leq 
\liminf \int_0^T C(q^n_s,u^n_s)ds$ and by continuity of the $g_k$'s
and of $E$, 
$J(q^\infty_0,u^\infty)\leq \lim J(q^n_0,u^n)=\inf J$.
\end{proof}
  Obviously, this proof gives also the existence of a minimizer in 
the \emph{exact} case where the spline is constrained to go through a sequence 
$(x^D_{t_k})_{1\leq k\leq M}$ of landmarks configurations if there exists at least 
one such controlled path with finite cost. 
\subsection{Directional derivatives and  Euler-Lagrange Equation} \label{AdjointLagrange}
\begin{prop}
  Assume ($\mathrm{H0}$) and let $u,\delta u\in L^2([0,T],\mathbb{R}^{nd})$ and
  $q_0,\delta q_0\in \mathbb{R}^{nd}\times \mathbb{R}^{nd}$. For
  any $\epsilon>0$, we denote $q_{t,\epsilon}\in
  C([0,T],\mathbb{R}^{nd}\times \mathbb{R}^{nd})$, the solution of
  (\ref{eq:23.2.1}) for the control $u_{t,\epsilon}\doteq
  u_t+\epsilon\delta u_t$ and the initial condition $q_{0,\epsilon}=
  q_0+\epsilon \delta q_0$. Then we have
\begin{equation}
\lim_{\stackrel{\epsilon\neq 0}{\epsilon\to 0}}\sup_{t\leq
  T}|\frac{q_{t,\epsilon}-q_t}{\epsilon}-\delta q_t|=0\label{23.2.12}
\end{equation}
where $\delta q$ is the absolutely continuous solution on $[0,T]$ of the 
linearized system
\begin{equation}
  \label{23.2.7}
  \dot{\delta q_t}=\partial_q f(q_t,u_t)\delta q_t+\partial_u f(q_t,u_t)\delta u_t
\end{equation}
with initial condition $\delta q_0$.
\label{Thm:23.2.2}
\end{prop}
\begin{proof}
  Assume $0<\epsilon\leq 1$ so that $\sup_{0<\epsilon\leq 1}\int_0^T|u_{s,\epsilon}|^2dt\leq
  2\int_0^T|u_{s}|^2+|\delta u_s|^2ds<\infty$. As we did in the proof of
  Proposition \ref{prop:24.2.1}
 we deduce from  Proposition
  \ref{Thm:23.2.1} that $q_{t,\epsilon}\doteq (x_{t,\epsilon},p_{t,\epsilon})$
  is uniformly bounded by some $M>0$ for $(t,\epsilon)\in [0,T]\times
  ]0,1]$. Again, if $K_M>0$ is such that 
$$|f(q,u)-f(q',u')|\leq K_M|q-q'|+|u-u'|$$
for $|q|,|q'|\leq M$ and $u,u'\in \mathbb{R}^{nd}$ then $|q_{t,\epsilon}-q_t|\leq
K_M\int_0^t |q_{s,\epsilon}-q_s|ds+\epsilon(\delta q_0+\int_0^t|\delta u_s|ds)$
and by Gronwall's Lemma
\begin{equation}
  \label{23.2.8}
  |q_{t,\epsilon}-q_t|\leq\epsilon \exp(K_MT)(\delta q_0+\int_0^T |\delta u|_sds)\,.
\end{equation}
Since from ($\mathrm{H0}$), $f$ is $C^1$ and $\partial_qf(q,u)$ and
$\partial_uf(q,u)$ are uniformly bounded for $|q|\leq M$, there exists
a unique solution $\delta q$ of (\ref{23.2.7}) which is absolutely
continuous. Moreover,
\begin{align}
  \label{23.2.9}
  \begin{split}
    \big|\frac{q_{t,\epsilon}-q_t}{\epsilon} & -\delta q_t\big| \\
     & \leq
    \int_0^t\left|\frac{f(q_{s,\epsilon},u_{s,\epsilon})-f(q_s,u_s)}{\epsilon}-\partial_q
      f(q_s,u_s)\delta q_s+\partial_u f(q_s,u_s)\delta u_s\right|ds
\end{split}\\
    &\leq  \int_0^t \left|\partial_q
      f(q_s,u_s)(\frac{q_{s,\epsilon}-q_s}{\epsilon}-\delta q_s
      )\right|ds +\int_0^t\eta_{s,\epsilon}ds
\end{align}
where
\begin{eqnarray*}
\eta_{s,\epsilon}&\doteq &\left|\frac{f(q_{s,\epsilon},u_{s,\epsilon})-f(q_s,u_s)}{\epsilon}-\partial_q
    f(q_s,u_s)\frac{q_{s,\epsilon}-q_s}{\epsilon}+\partial_u
    f(q_s,u_s)\delta u_s\right|\\
&=&\left|\frac{f(q_{s,\epsilon},u_{s})-f(q_s,u_s)}{\epsilon}-\partial_q
    f(q_s,u_s)\frac{q_{s,\epsilon}-q_s}{\epsilon}\right|\,.
\end{eqnarray*}
However, from (\ref{23.2.8}) and the fact that $\partial_qf$ is
uniformly bounded for $|q|\leq M$, we get that $\eta_{s,\epsilon}$ is
uniformly bounded on $[0,T]\times ]0,1]$. Since $\eta_{s,\epsilon}\to
0$ for $\epsilon\to 0$, we get by Lebesgue's Dominated Convergence Theorem that $\int_0^T
\eta_{s,\epsilon}ds\to 0$. Using (\ref{23.2.9}) and Gronwall's Lemma,
we get
$$\left|\frac{q_{t,\epsilon}-q_t}{\epsilon}-\delta q_t\right|\leq
(\int_0^T\eta_{s,\epsilon}ds)\exp(\int_0^T \partial_qf(q_s,u_s)ds)\to 0\,.$$
\end{proof}
\begin{Thm}[Directional derivative]
  Assume ($\mathrm{H0}$), assume that $q\to{\mathcal K}_q\in{\mathcal M}_d(\mathbb{R})$
  is $C^1$, that $g_k$
  is $C^1$ for any $1\leq k\leq M$ and that $E$ is $C^1$. Let $C(q,u)=\frac{1}{2}\langle \mathcal{K}_qu,u\rangle$ and $T>t_M$. Then for any $u,\delta u\in
  L^2([0,T],\mathbb{R}^{nd})$ and $q_0,\delta
  q_0\in\mathbb{R}^{nd}\times \mathbb{R}^{nd}$, if 
$$J(\epsilon)\doteq E(q_{0,\epsilon})+\int_0^T C(q_{t,\epsilon},u_{t,\epsilon})dt
+\sum_{k=1}^Mg_k(q_{t_k,\epsilon})$$ we have
\begin{align}
\begin{split}
\lim_{\epsilon\to
  0}&\frac{J(\epsilon) -J(0)}{\epsilon} \\ 
& =\ \langle \nabla E,\delta q_0\rangle+\int_0^T\left(\partial_qC(q_s,u_s)\delta
q_s+\partial_uC(q_s,u_s)\delta u_s\right)ds+\sum_{k=1}^M \langle \nabla g_k(q_{t_k}),\delta
q_{t_k}\rangle
\end{split}
  \label{23.2.10a}\\
& =\ \langle \nabla E(q_0)+P_0,\delta q_0\rangle+\int_0^T\langle
\nabla_uC(q_s,u_s)+\partial_uf^T(q_s,u_s)P_s,\delta u_s\rangle ds\label{23.2.10}
\end{align}
where $C(q,u)\doteq \langle {\mathcal K}_qu,u\rangle$, $\delta q_t$ is
solution of (\ref{23.2.7}) and $P_t$ is of bounded variations with
$P_T=0$ and
\begin{equation}
dP_t=-\partial_qf(q_t,u_t)^TP_tdt-\sum_{k=1}^M \nabla
  g_k(q_{t_k})\otimes \delta_{t_k}\label{23.2.14}
\end{equation}
\label{Thm:23.2.4}
where $v\otimes \delta_x$ denotes a vectorial Dirac measure at
location $x$ with value $v$.
\end{Thm}
\begin{proof}
For any variation $\delta u$ of the control $u$, we get by Proposition
\ref{Thm:23.2.2} that $\epsilon\to q_{t,\epsilon}$ is differentiable 
and $\partial_\epsilon q_{|\epsilon=0}=\delta q$ where
  $\dot{\delta q}=\partial_qf(q_s,u_s)\delta q_s+\partial_u f(q_s,u_s)\delta
  u_s$. Moreover,
  \begin{align*}
    \label{23.2.11}
    \begin{split}
      A(\epsilon) &\doteq 
      \frac{J(\epsilon)-J(0)}{\epsilon}\\
&\quad -\left(\langle \nabla
        E(q_0),\delta q_0\rangle+\int_0^T(\partial_qC(q_s,u_s)\delta
        q_s +\partial_uC(q_s,u_s)\delta u_sds+\sum_{k=1}^M\langle
        \nabla g_k(q_{t_k}),\delta q_{t_k}\rangle\right)
    \end{split}
\\
& =  \int_0^T\partial_qC(q_s,u_s)\left(
  \frac{q_{s,\epsilon}-q_s}{\epsilon}-\delta
  q_s\right)ds\\
&\quad +\sum_{k=1}^M\langle \nabla
g_k(q_{t_k}),\frac{q_{t_k,\epsilon}-q_{t_k}}{\epsilon}-\delta
q_{t_k}\rangle+\zeta_{\epsilon}+\int_0^T \eta_{s,\epsilon}ds
  \end{align*}
where
$$\zeta_\epsilon=\sum_{k=1}^M\left(\frac{g_k(q_{t_k,\epsilon})-g_k(q_{t_k})}{\epsilon}-\langle
\nabla
g_k(q_{t_k}),\frac{q_{t_k,\epsilon}-q_{t_k}}{\epsilon}\rangle\right)+\frac{E(q_{0,\epsilon})-E(q_0)}{\epsilon}-\langle
\nabla E(q_0),\delta q_0\rangle$$
and
$$\eta_{s,\epsilon}
=\frac{C(q_{s,\epsilon},u_{s,\epsilon})-C(q_s,u_s)}{\epsilon}-\left(\partial_qC(q_s,u_s)\frac{q_{s,\epsilon}-q_s}{\epsilon}+\partial_uC(q_s,u_s)\delta
  u_s\right)\,.$$
From the fact that the $g_k$'s and $E$ are $C^1$ and (\ref{23.2.8}), we get
that $\zeta_\epsilon\to 0$ with $\epsilon\to 0$. Since $q_{s,\epsilon}$ is uniformly bounded for $(s,\epsilon)\in
[0,T]\times ]0,1]$ and $\mathcal K_q$ is $C^1$, one easily gets from
(\ref{23.2.8}) that there exist $a,b>0$ such that $|\eta_{s,\epsilon}|\leq
a+b(|u|_s^2+|\delta u_s|^2)$. Since $\eta_{s,\epsilon}\to 0$ as
$\epsilon\to 0$, we get by Lebesgue's Dominated Convergence Theorem that
$\int_0^T \eta_{s,\epsilon}ds\to 0$. Using (\ref{23.2.12}), we get
eventually $A(\epsilon)\to 0$ so that 
\begin{equation}
  \begin{split}
    \delta J\doteq \frac{d J}{d\epsilon}(0) & =\langle \nabla
    E(q_0),\delta q_0\rangle\\
&\quad +\int_0^T\left(\partial_qC(q_s,u_s)\delta
      q_s +\partial_uC(q_s,u_s)\delta u_s\right)ds+\sum_{k=1}^M \langle
    \nabla g_k(q_{t_k}),\delta q_{t_k}\rangle
  \end{split}
\label{23.2.13}
\end{equation}
and (\ref{23.2.10a}) is proved.
Introducing now $M_{t,s}$ the semi-group solution of $\partial_s
M_{t,s}=\partial_q f(q_s,u_s) M_{t,s}$ with $M_{t,t}=\text{Id}_{2nd}$ we
get
$$ \delta q_s=\int_0^s M_{t,s}\,\partial_uf(q_t,u_t) \delta u_t
dt+M_{0,s}\delta q_0\,.$$
and from (\ref{23.2.13})
\begin{equation}
  \label{eq:3}
  \begin{split}
    \delta J &=\langle \nabla E(q_0),\delta
    q_0\rangle+\int_0^T\langle\nabla_uC(q_t,u_t) ,\delta u_t\rangle dt\\
    &\quad +\int_0^T\langle\nabla_qC(q_s,u_s),\delta q_s\rangle ds
 + \sum_{k=1}^n \langle \nabla g_k(q_{t_k}),\delta q_{t_k}\rangle\,.
  \end{split}
\end{equation}
Thus we have
\begin{eqnarray*}
\delta J &=&\int_0^T\langle\nabla_uC(q_t,u_t)  + \partial_u
f(q_t,u_t)^T \int_t^T M_{t,s}^T\nabla_qC(q_s,u_s)\,ds\\
&&+ \sum_{k=1}^n \partial_u f(q_t,u_t)^T
{M_{t,t_k}}^T\nabla g_k(q_{t_k})\mathbf{1}_{t\leq t_k},\delta
u_t\rangle dt\\
&&+\langle \nabla E(q_0)
+\int_0^T
M_{0,s}^T\nabla_qC(q_s,u_s)\,ds+\sum_{k=1}^nM_{0,t_k}^T\nabla
g_k(q_{t_k})\mathbf{1}_{t\leq t_k},\delta q_0\rangle\\
& = & \int_0^T\langle\nabla_uC(q_t,u_t)  + \partial_u
f(q_t,u_t)^T P_t,\delta
u_t\rangle dt+\langle \nabla E(q_0)+P_0,\delta q_0\rangle\,.
\end{eqnarray*}
where
$$P_t\doteq \int_t^T
M_{t,s}^T\nabla_qC(q_s,u_s)\,ds+\sum_{k=1}^nM_{t,t_k}^T\nabla
g_k(q_{t_k})\mathbf{1}_{t\leq t_k}\,.$$
One easily sees that $P_t$ is absolutely continuous between the
observation times with jumps at the observation times. Moreover,
differentiating $t\to M_{t,t'}M_{t',t}=\text{Id}$, we get $\partial_t
M_{t,t'}=-M_{t,t'}\partial_qf(q_t,u_t)$ so that we get (\ref{23.2.14}). 
\end{proof}
  From Thm \ref{Thm:23.2.4}, we get immediately the Euler-Lagrange 
equations for shape splines evolution~: 
\begin{equation}
\left\{
  \begin{array}[h]{l}
  \label{eq:8.3.1}
\dot{q}_t = f(q_t,u_t)\\\\
  \mathcal{K}_{q_t}u_t+P^p_t
  =0\\\\
dP_t=-\partial_qf(q_t,u_t)^TP_tdt-\sum_{k=1}^M \nabla
  g_k(q_{t_k})\otimes \delta_{t_k}\\\\
\end{array}\right.
\end{equation}
where $P_t=
\begin{pmatrix}
  P^x_t & P^p_t
\end{pmatrix}^T
$ with boundary conditions
\begin{equation}
 P_T=0,\ \nabla E(q_0)+P_0=0\,. 
\end{equation}
\begin{rem}
\label{rem:9.3.1}
  Note that in our experiments, we will consider that $x_0$ is fixed
  and let $p_0$ be free (sometimes called \emph{natural spline} in the
  classical cubic splines framework). This corresponds to a flat prior
  on $p_0$ and gives the new boundary conditions
  \begin{equation}
    P_T=0,\ P^p_0=0\,. 
  \end{equation}
Moreover, the cost functions $g_k$ will usually not depend on $p$ so that 
$\partial_pg_k(q)=0$ and $P^p_t$ is absolutely continuous in time. In particular, 
we get in this case that for natural splines, $u_t$ is continuous and vanishes 
at the boundary of the interval $[0,T]$.
\end{rem}
\section{Numerical experiments}
\label{sec:numexp}
In this section we provide preliminary experiments illustrating the behavior 
of shape splines in simple 2D synthetic experiments. 
\begin{figure}[phbt]
  \centering
\begin{tabular}[h]{ll}
 \hspace{-15mm}  
\begin{minipage}[h]{0.45\linewidth}
    \includegraphics[width=8cm]{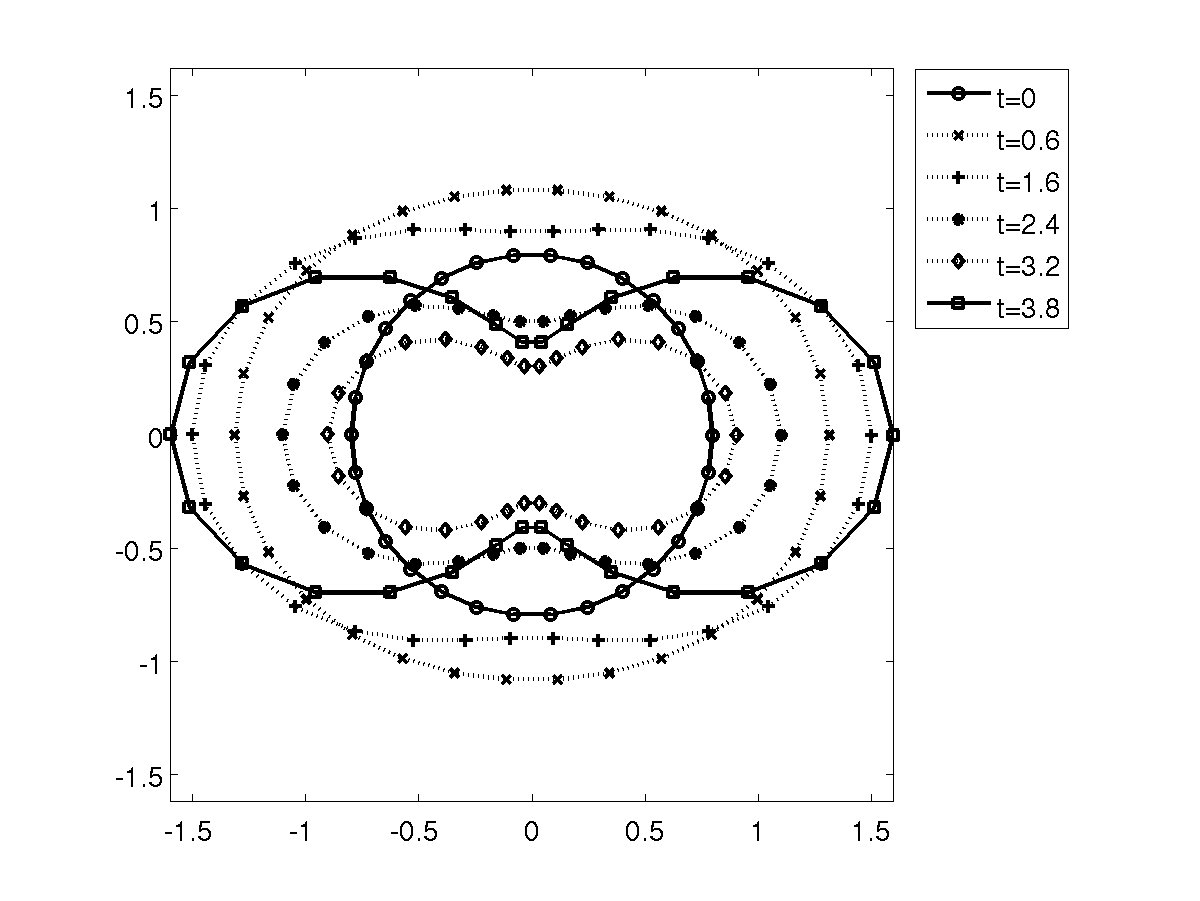}
    \end{minipage} &
    \begin{minipage}[h]{0.5\linewidth}
     \hspace{-0mm} \includegraphics[width=10cm]{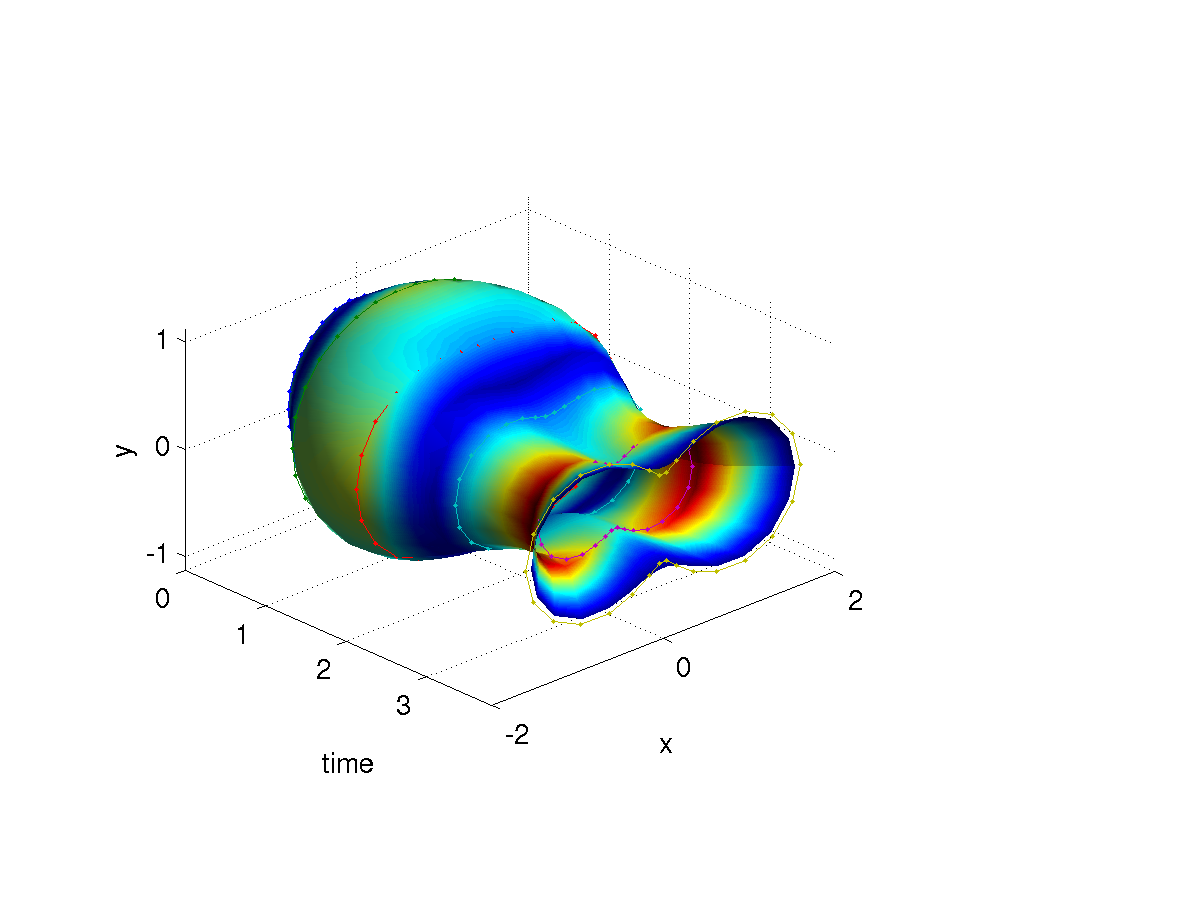}
    \end{minipage}
\end{tabular}
  \caption{Left hand side :2D plots of the sequence of sampled curves $c_k$ for $k=0,\cdots, 5$. The initial curve is a circle and the final shape is a horizontally pinched ellipse (plain curve). The intermediate curves  are generated by linear interpolation between the initial and final shapes. Note that no noise is added here and that the intermediate shapes are rescaled by a factor $r_k$ depending on $k$. Right hand side : estimated shape spline displayed in space-time representation. The color indicates the value of the norm of $u$ on the surface through time (blue for low values, red for high values)}
\label{fig:10.2.1}
\end{figure}
We start with a small family of observation times 
$0\leq t_1<\cdots <t_M\leq T$ and for each time, a landmark configuration 
$x^D_{t_k}\in\mathbb{R}^{2n}$ defined as a noisy version of regularly 
sampled curves $c_k$. To stay 
close to some realistic framework where the observation points are sparse, 
the number $M$ of time points is kept small to $M=5$ (with $t_5=3.8$). In these 
experiments, we focus on the simple Euclidean metric $|x_{t_k}^D-x_{t_k}|^2$ 
in the data term, which is in agreement with our Gaussian 
noise assumption.

In our first experiment, we consider 
the evolution of an initial 2D circular shape which somewhat linearly 
evolves into a pinched ellipso\"idale shape. The shapes are regularity 
sampled in a consistent way so that the problem of point correspondence 
between two time points does not need to be consider (see Fig 
\ref{fig:10.2.1}). To emphasize the space-time regularity provided by 
the spline shapes interpolation, we display the result as a surface 
in the 3D space-time. The first interesting point to note is that 
the shape spline provides an actual smooth interpolation of the evolution
in time between the observation epochs and extends to shape spaces the specific behavior of classical cubic splines. In particular, the shape spline actually joins the observed shapes with very good accuracy despite the fact that we are using
here inexact matching.  
\begin{figure}[h]
  \centering
  \begin{tabular}[h]{ccc}
    \hspace{-10mm}\begin{minipage}[h]{0.4\linewidth}
      \includegraphics[width=5.5cm]{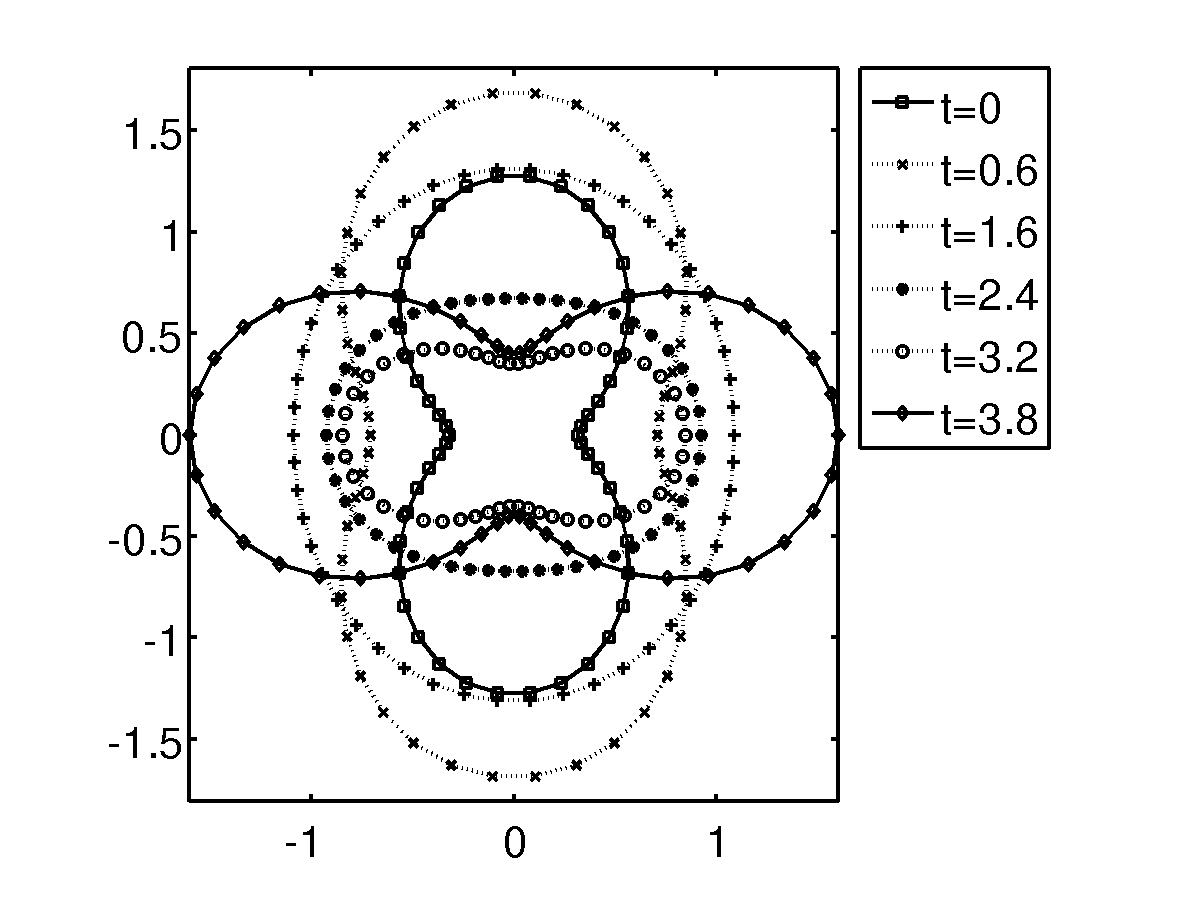}
    \end{minipage}
&
 \hspace{-7mm}  \begin{minipage}[h]{0.35\linewidth}
      \includegraphics[width=6cm]{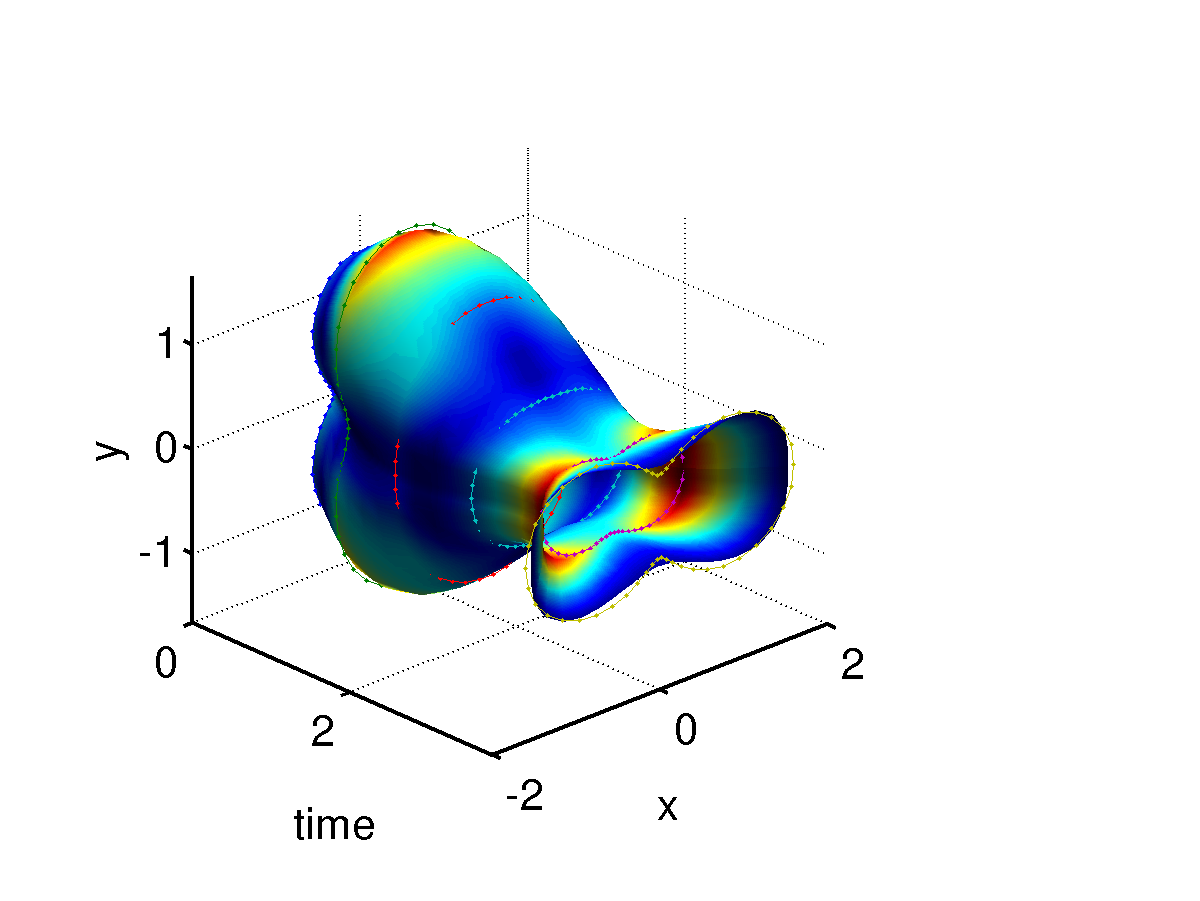}
    \end{minipage}

&
\begin{minipage}[h]{0.3\linewidth}
\includegraphics[width=6cm]{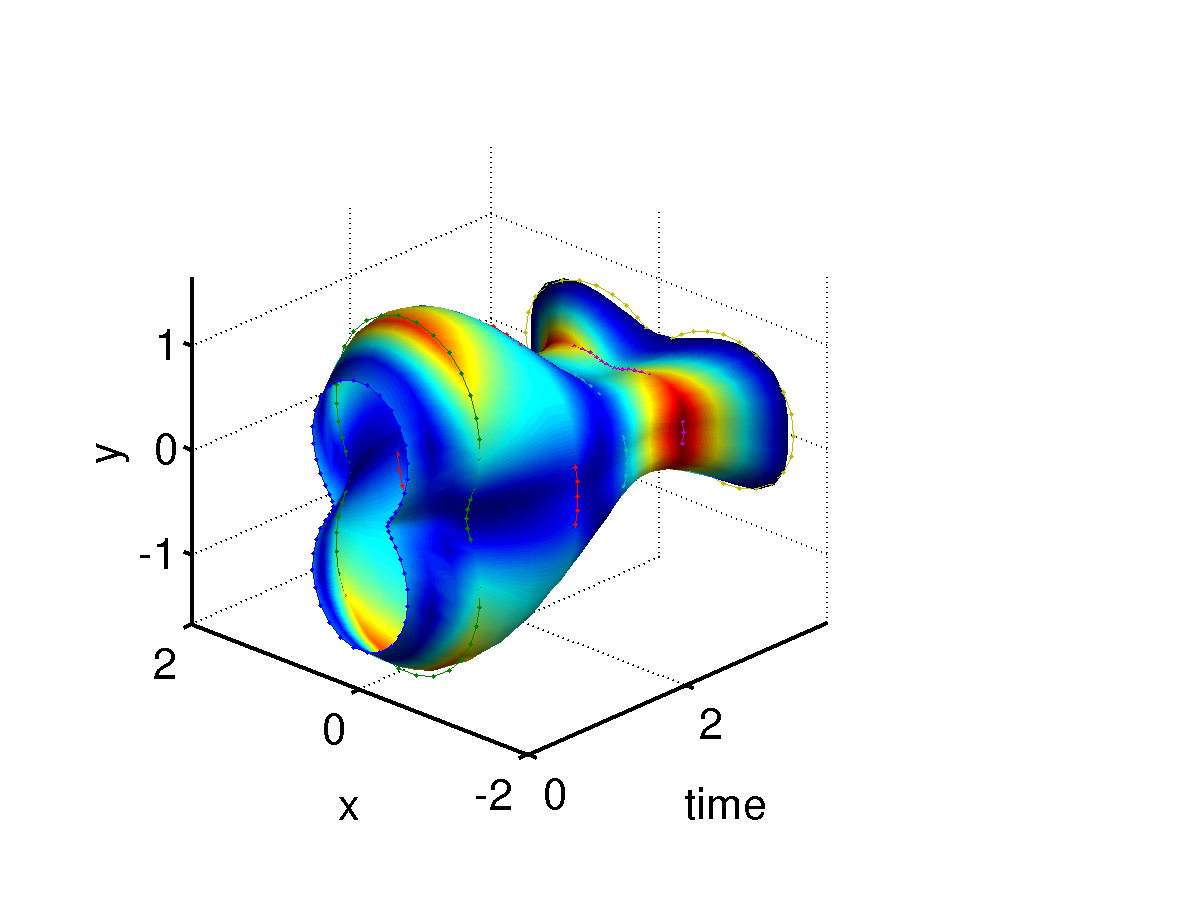}
\end{minipage}

\\ 
 \hspace{-10mm}\begin{minipage}[h]{0.4\linewidth}
   \includegraphics[width=5.5cm]{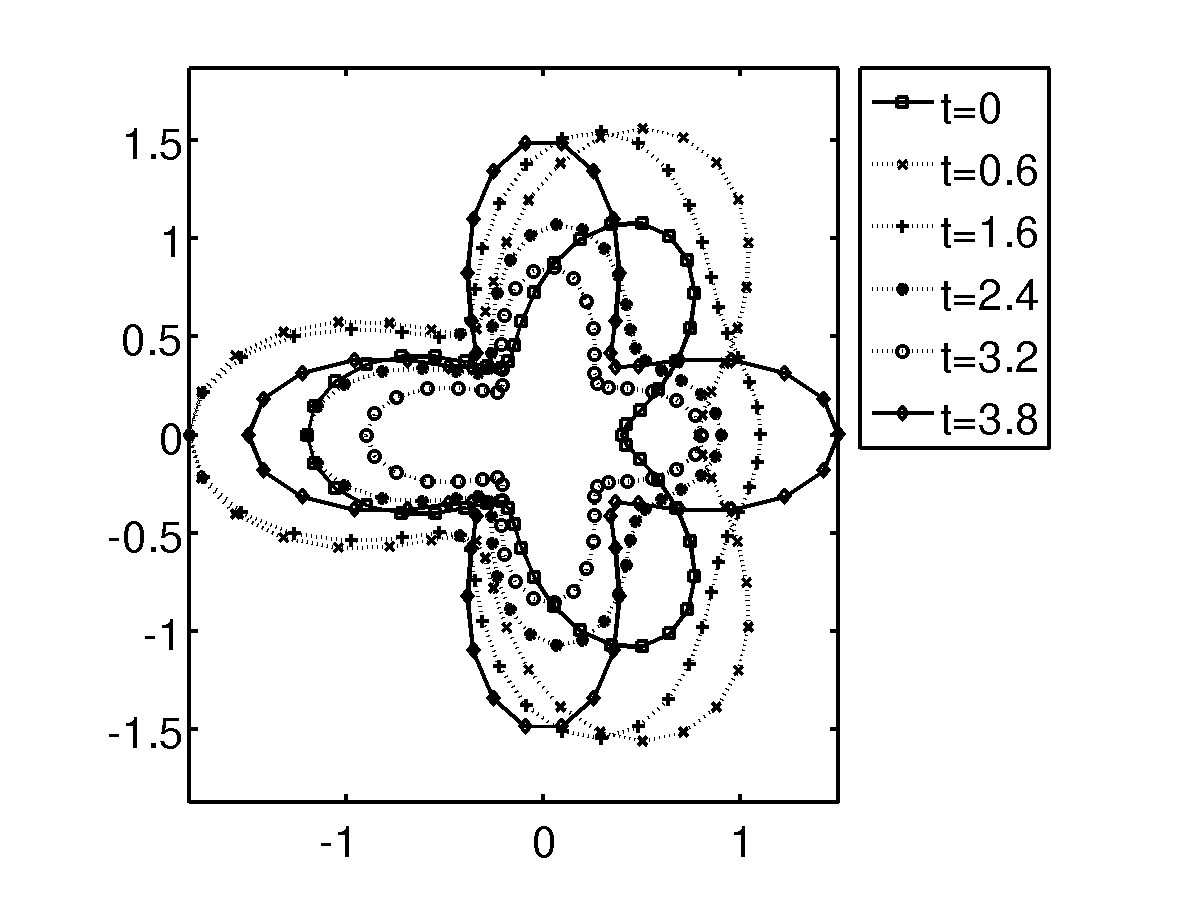}
\end{minipage}
&
 \hspace{-7mm}\begin{minipage}[h]{0.35\linewidth}
      \includegraphics[width=6cm]{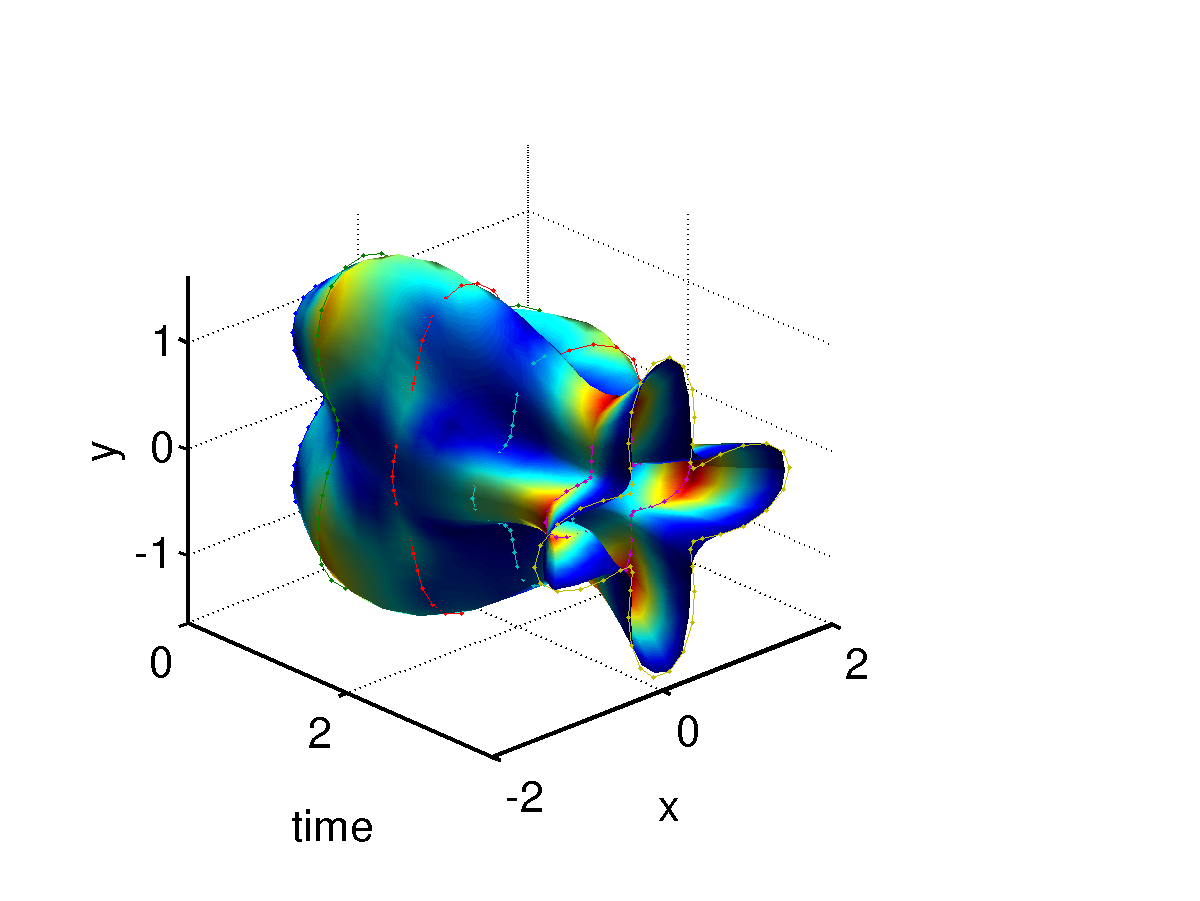}
    \end{minipage} 
&
\begin{minipage}[h]{0.3\linewidth}
\includegraphics[width=6cm]{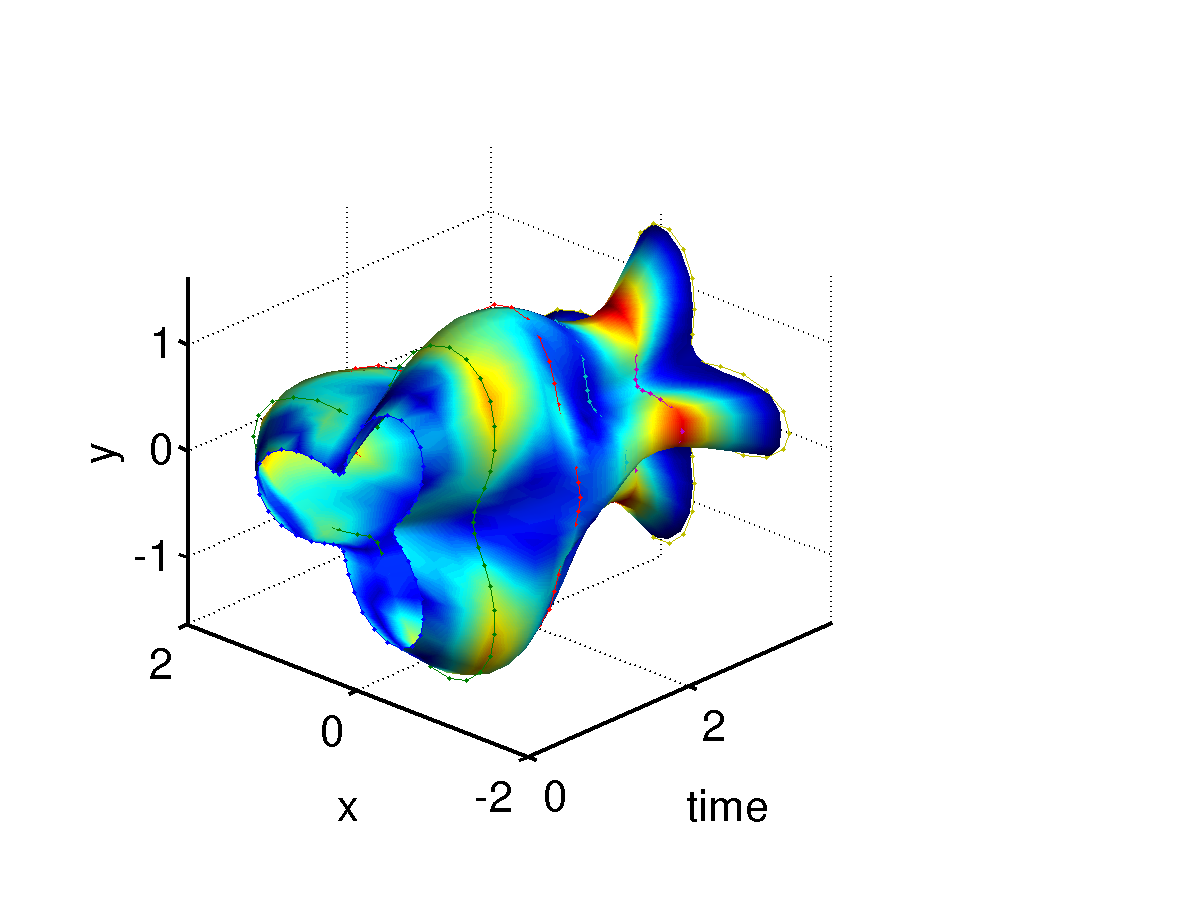}
\end{minipage}
\end{tabular}
  \caption{The figure displays two estimation experiments of shape
    splines for two different problems. From left to right, the first
    column displays the sequence of observed shapes, the second and
    the third column two different viewpoints. The color of any point 
    on the space-time surface is related
    to the norm of its control variable (with increasing values from
    deep blue to red)}
\label{fig:14.1.2}
\end{figure}
A second important point is that the shape spline comes with 
the estimation in time of the control variable $u$ which can be
interpreted as an external force bending the underlying geodesic. 
What we see in this example and in the other similar 
situations, displayed in Fig \ref{fig:14.1.2}, is that the point-wise values of $u$
give interesting information on the evolution process. More
specifically, in every example, the initial and end shape 
have specific features (numbers of lobes, orientation, etc)  
and the most active zones are easily interpretable and correspond to 
transition regions in the shape evolution.
\subsection{Robustness to noise}
The robustness to noise is a rather important subject from a 
practical point of view. Indeed, the noise basically degrades the \emph{spatial}
resolution of the measurements so that the evolution through time 
of a particular point of the evolving curve may be a sharply broken line. The 
standard spline approach can be quite efficient in filtering this noise if 
the time sampling frequency is high enough.  This is hardly the case in 
many important situations. However, neighboring points behave coherently 
through time and offer an interesting source of spatial redundancy. 
\begin{figure}[h]
  \centering
  \begin{tabular}[h]{cc}
 \hspace{-50mm} 
    \begin{minipage}[h]{0.5\linewidth}
    \includegraphics[width=9cm]{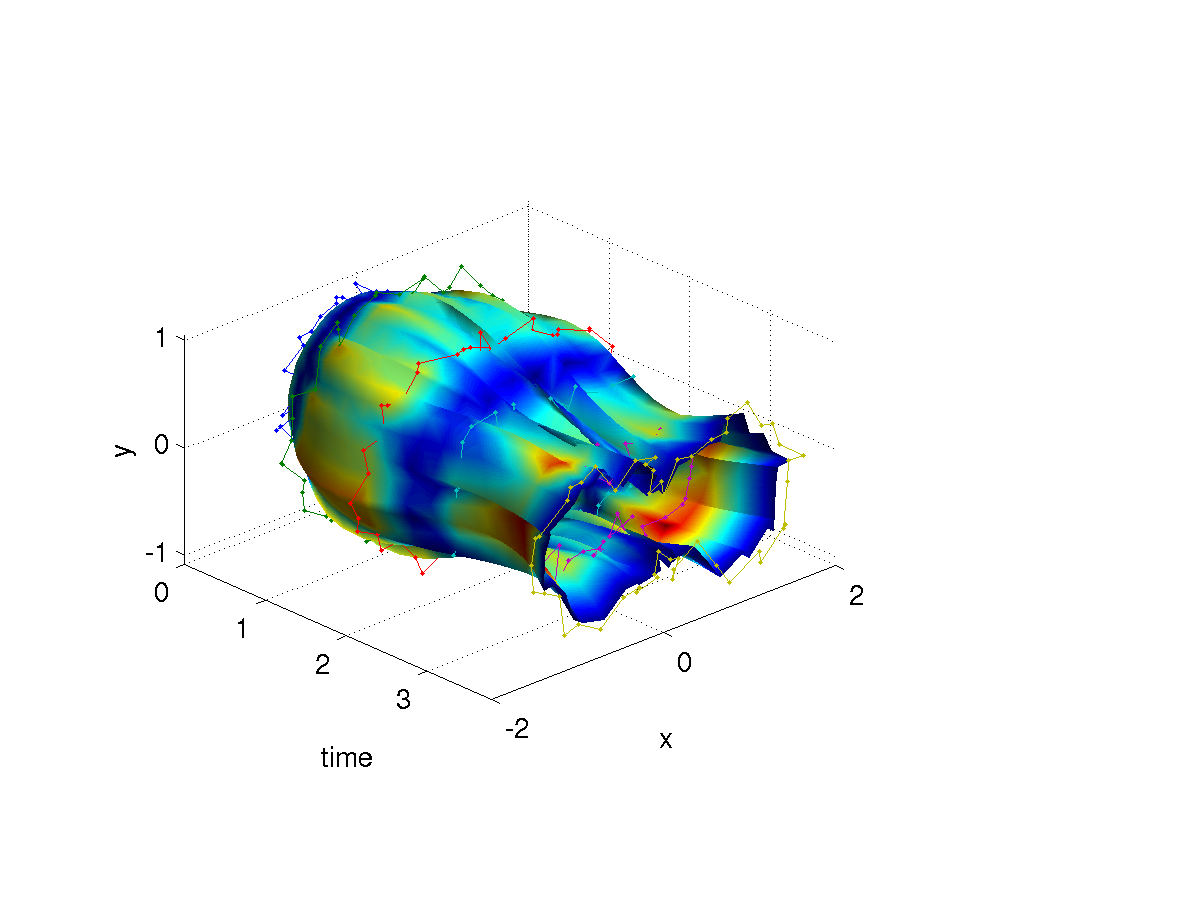}
    \end{minipage}
%&
%\begin{minipage}[h]{0.3\linewidth}
%\hspace{-10mm}\includegraphics[width=8cm]{Paperfig/vasque_w2s4lam2n2c.png}
%\end{minipage}
&\hspace{-10mm}
\begin{minipage}[h]{0.3\linewidth}
\includegraphics[width=9cm]{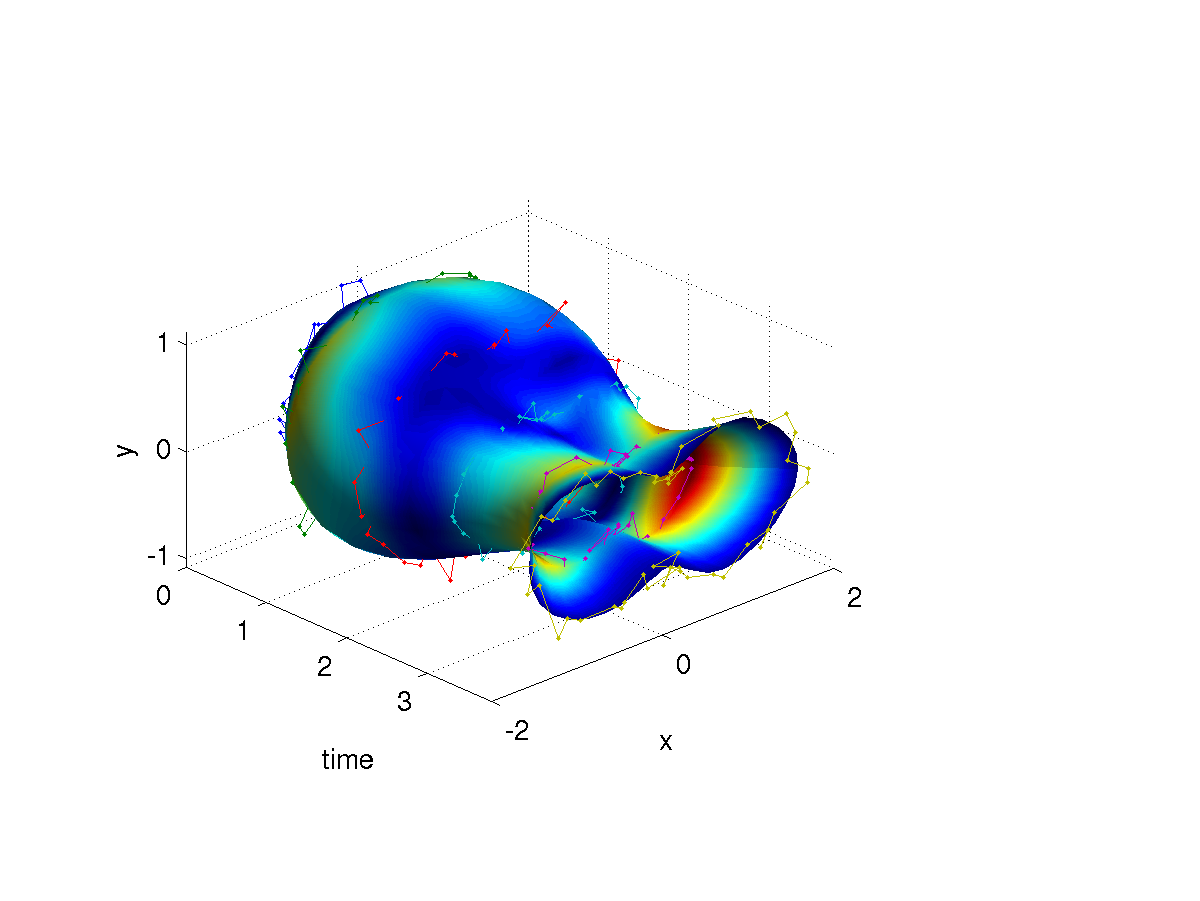}
\end{minipage}    
  \end{tabular}
  \caption{Robustness to noise. An i.i.d. Gaussian noise with standard 
deviation $\sigma=0.1$ is added to each measurement point. Shape splines are computed with two different scale parameters, $\lambda=0.001$ on the left, $\lambda=0.6$ on the right.}
  \label{fig:10.2.2}
\end{figure}
Much of the large deformation shape space theory involves the integration 
of spatial redundancy in the comparison between shapes. The shape spline 
setting, considering shapes as a whole and not as a bag of independent points,
keeps this important aspect but adds a time component and considers the problem 
in the full space-time setting. This robustness to noise is illustrated in Figure 
\ref{fig:10.2.2} where an i.i.d. Gaussian noise with standard deviation 
$\sigma=0.1$ is added and a series of shape splines are computed under 
increasing values of the spatial regularity scale parameter $\lambda$ as 
introduced in 
 (\ref{eq:15-4}). For low values of this parameters (with respect to the 
overall scale of the shapes) the reconstructed evolution is clearly far 
from any reasonable solution since the spatial redundancy is hardly taken 
into account.  Increasing the value of $\lambda$ to values in accordance 
to the scale of the object produces a much better reconstruction of the 
actual shapes at any observation time but also keeps existing time regularity.
\subsection{Extrapolation}
Another distinguished feature of the usual spline setting which is 
extended in the shape spline setting is the fact that the extrapolation 
of the data outside the interval of observation is quite straightforward. 
\begin{figure}[h]
  \centering
  \begin{tabular}[hbtp]{cc}
\hspace{-10mm}
    \begin{minipage}[h]{0.6\linewidth}
      \includegraphics[width=7cm]{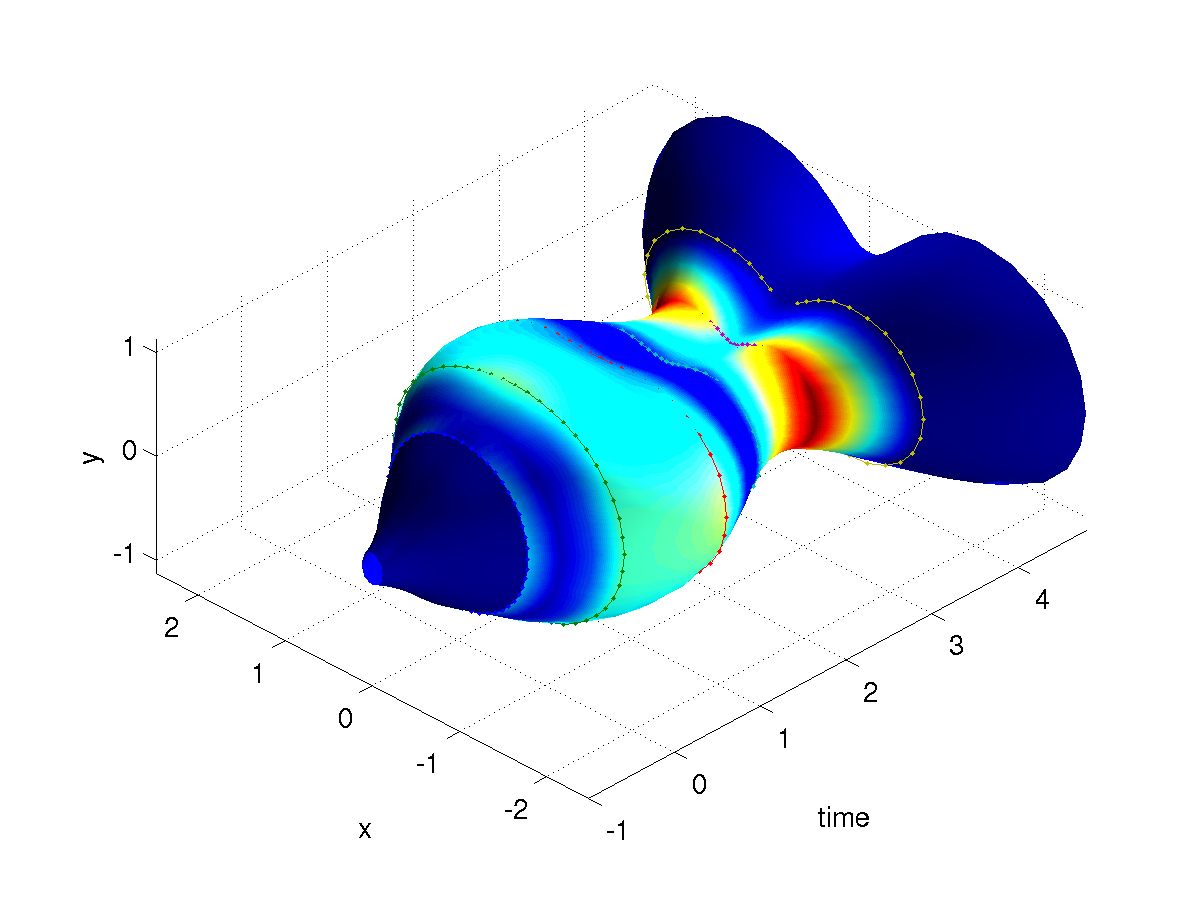}
    \end{minipage}
&
\begin{minipage}[h]{0.5\linewidth}
\hspace{-10mm}\includegraphics[width=7cm]{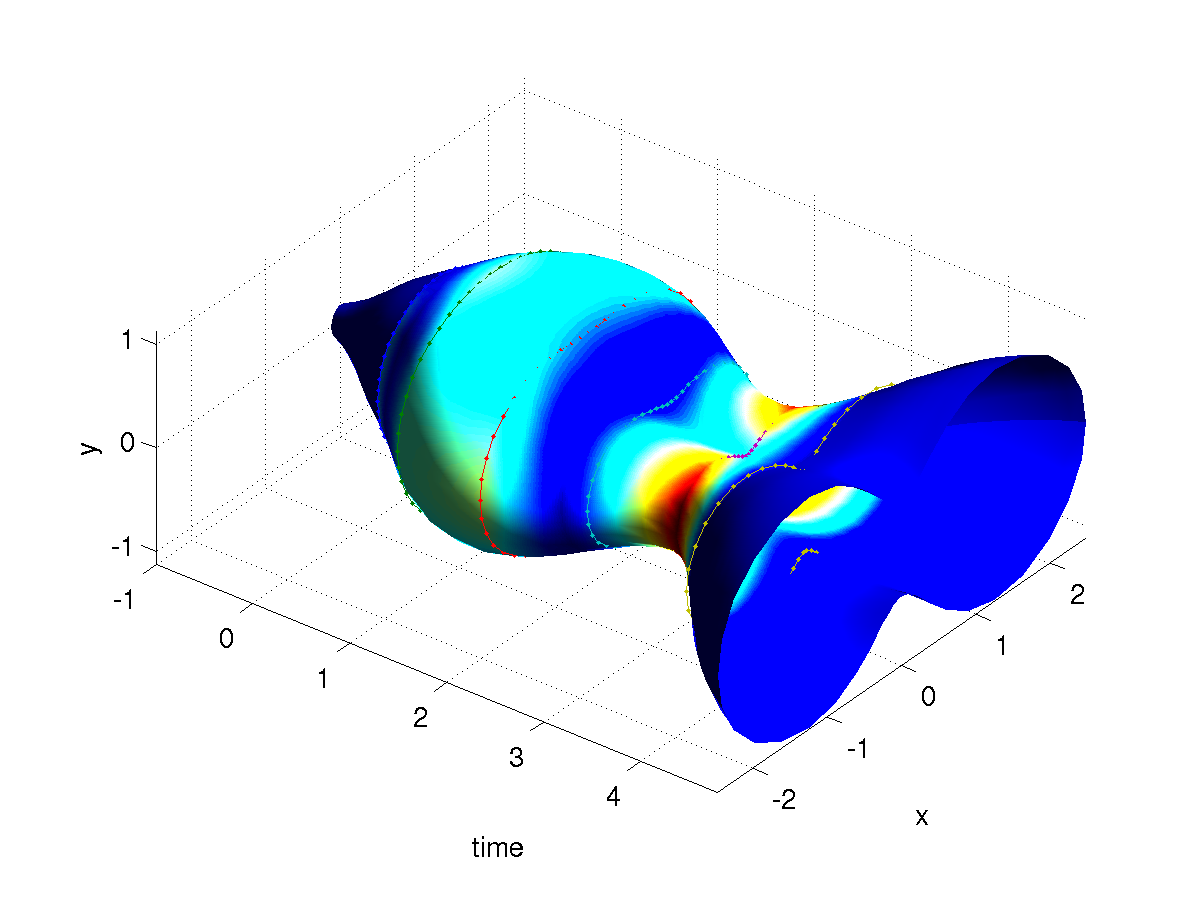}
\end{minipage} 
  \end{tabular}
  \caption{Extrapolations. The extrapolation of the 
evolution at both ends of the observation interval ($\lambda=0.6$, $\sigma=0$).}
  \label{fig:10.2.3}
\end{figure}
Indeed, outside the limits of the observation interval, the value of the 
control parameter $u$ is set to zero and 
the evolution is naturally extended with a geodesic evolution. Moreover, one 
can check (see Remark \ref{rem:9.3.1}) that $u$ vanishes at the last observation time 
so that the previous extension is $C^0$ for the control variable $u$ 
and $C^1$ for the shape variable $x$. Note that in the standard growth 
model described in (\ref{eq:21.12.1}), the evolution is extrapolated by 
fixing the shape variable to its value at the last observation and this 
extrapolation is only $C^0$. In Figure \ref{fig:10.2.3}, we  
display a simple example of extrapolation where the underlying evolution 
within the observation interval is the same than in Figure \ref{fig:10.2.1}. 
The computed 
extrapolation appears visually quite natural at both ends.
\subsection{Comparison with piecewise geodesic evolution}
We end this section with a comparison with the piecewise geodesic interpolation scheme \cite{mty02},  derived from (\ref{eq:21.12.1}) for a finite set 
of observation points as the minimizing solution of 
\begin{equation}
  \label{eq:14.2.1}
  J^x(v)=\frac{1}{2}\int_{0}^{T}|v_t|_V^2dt +\gamma\sum_{i=1}^M |x_{t_k}^D-x_{t_k}|^2
\end{equation}
where as previously stated, $x_t=\phi^v_t(x_0)$ and $\phi^v_t$ is the flow of 
$v\in L^2([0,1],V)$. 
\begin{figure}[h]
  \centering
      \begin{tabular}[hbtp]{ccc}
\hspace{-10mm}
        \begin{minipage}[h]{\jmcm}
          \includegraphics[width=\imcm]{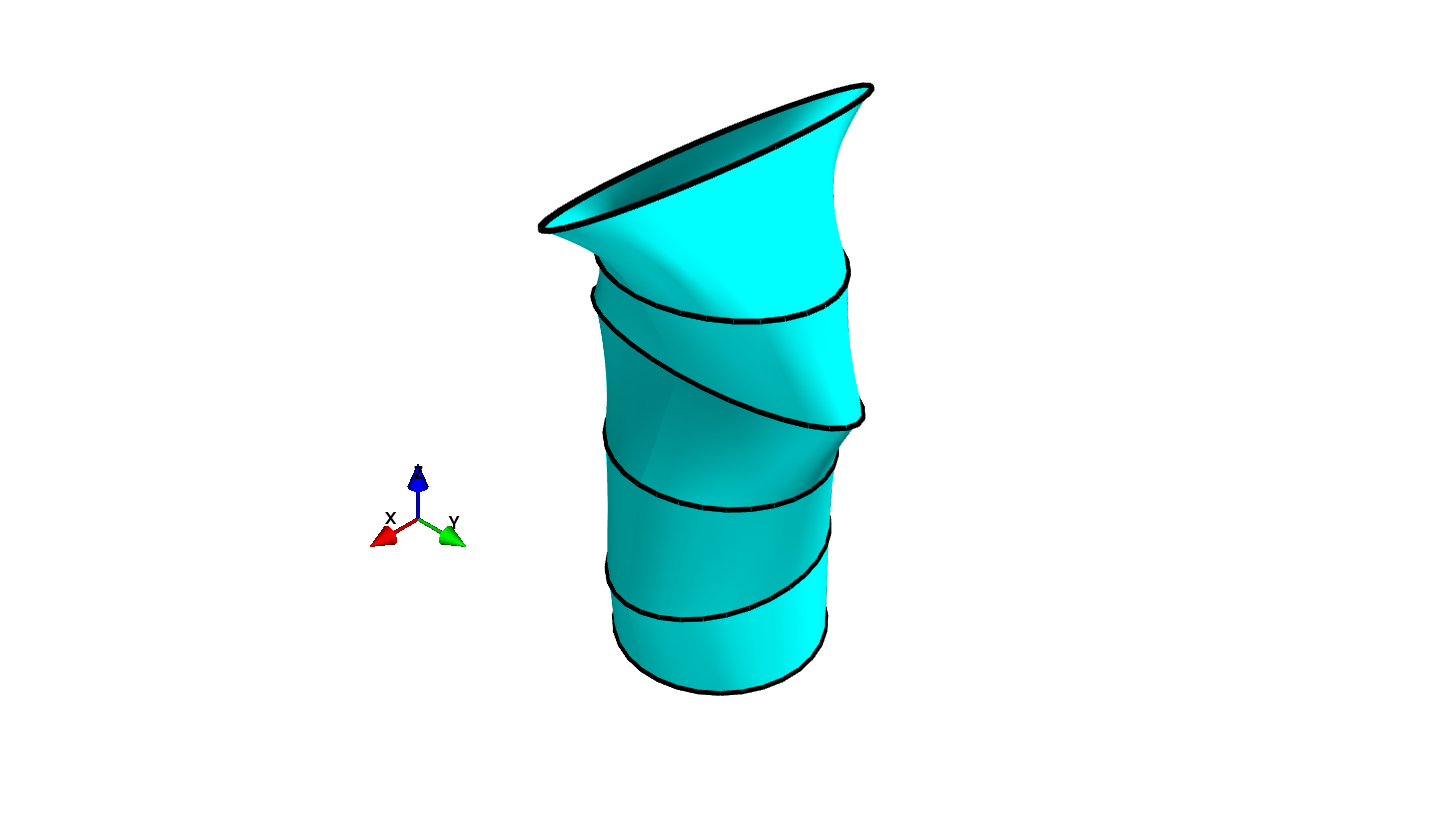}
        \end{minipage}&
        \begin{minipage}[h]{\jmcm}
          \hspace{-5mm} \includegraphics[width=\imcm]{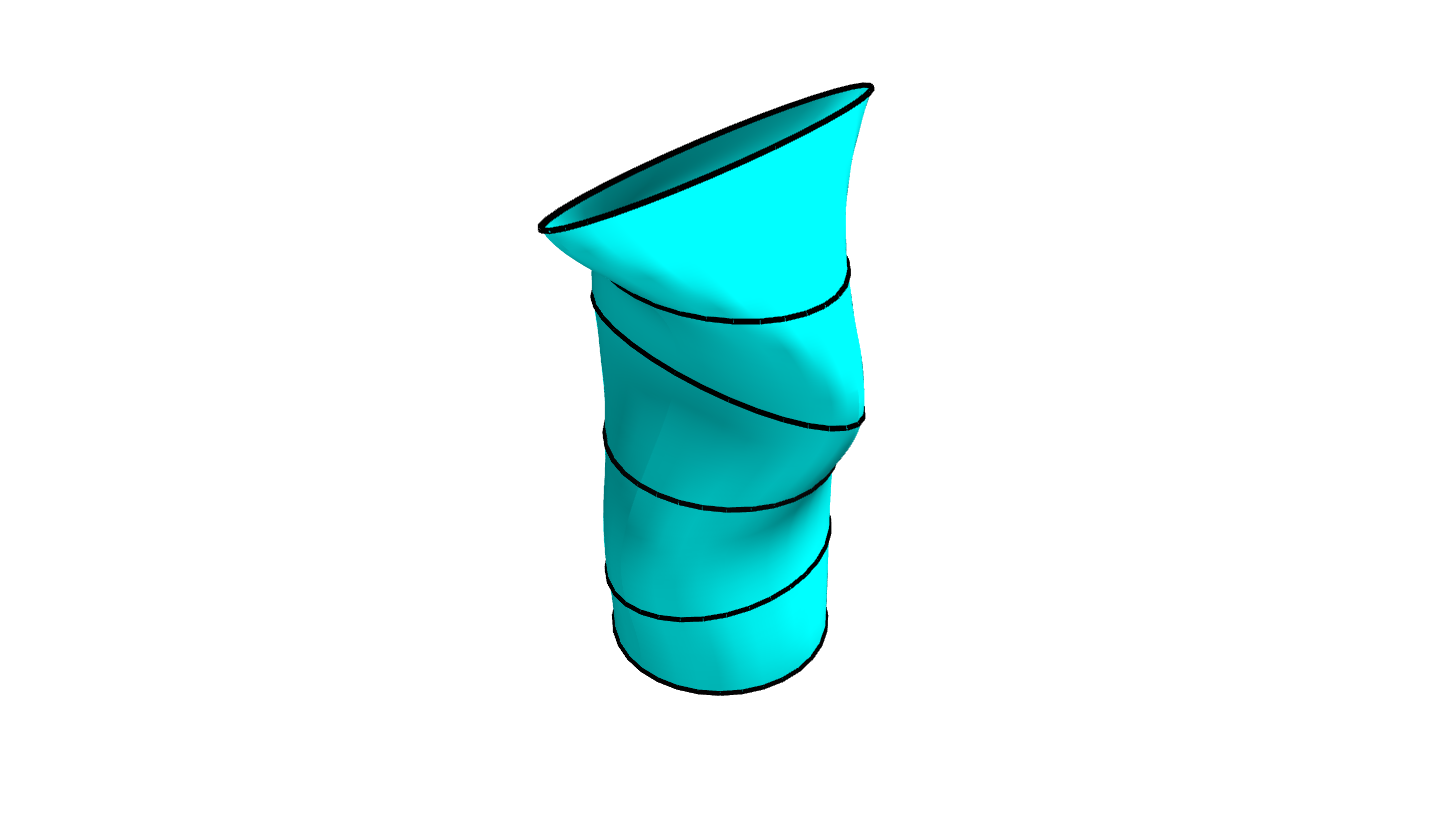}
        \end{minipage}&
        \begin{minipage}[h]{\jmcm}
          \hspace{-10mm} \includegraphics[width=\imcm]{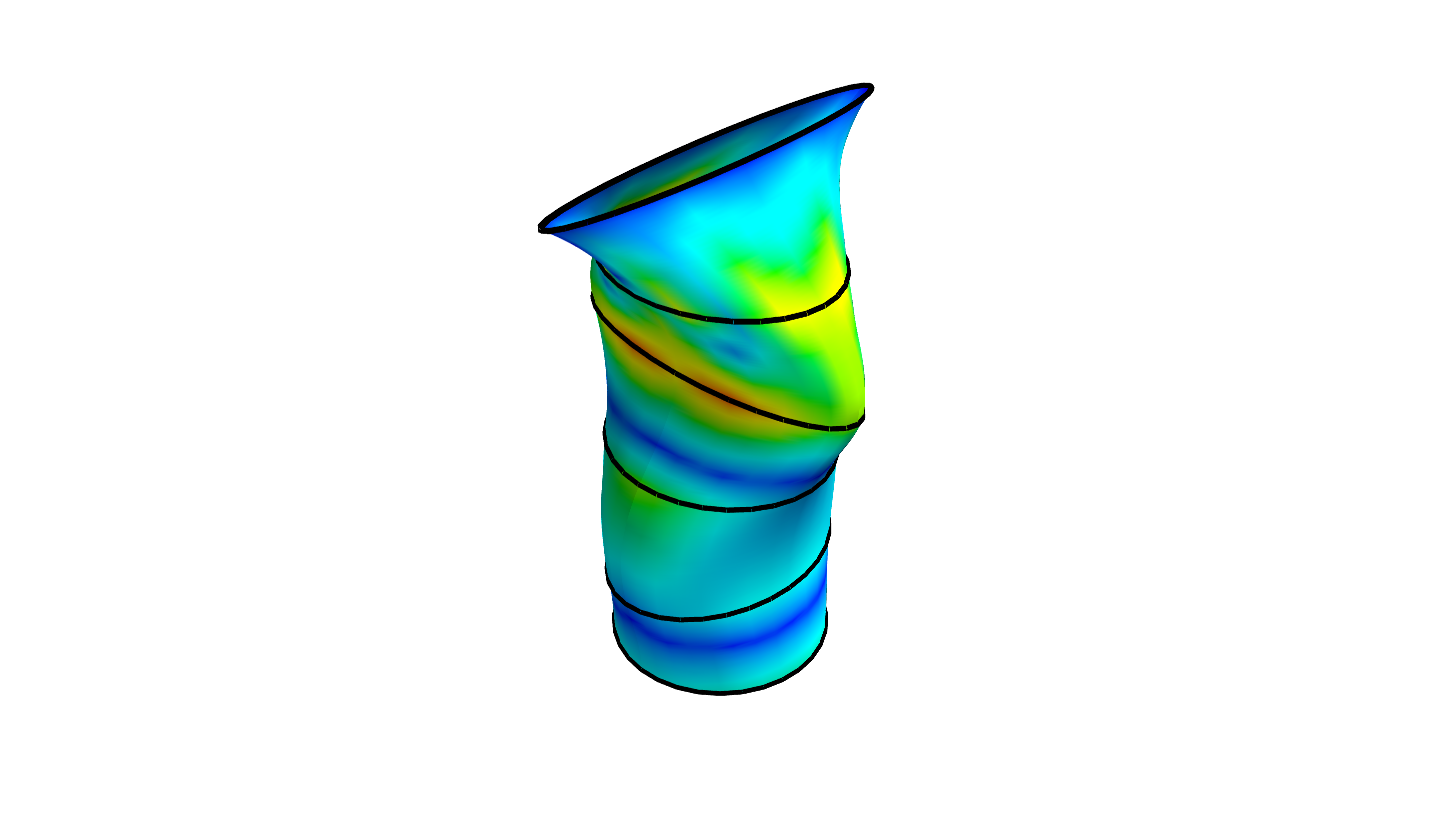}
        \end{minipage}\\
 \hspace{-10mm}       \begin{minipage}[h]{\jmcm}
          \includegraphics[width=\imcm]{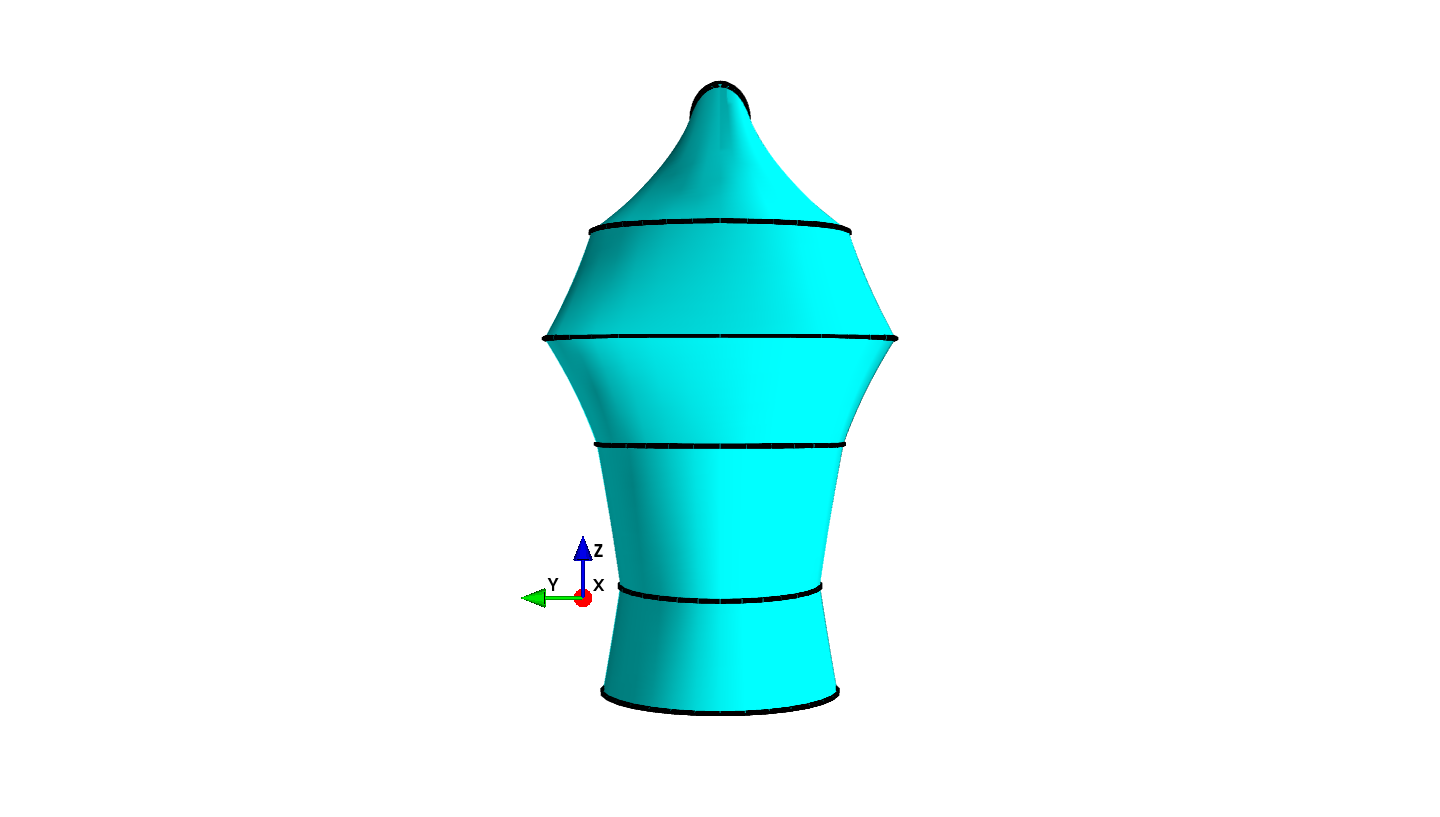}
        \end{minipage}&
        \begin{minipage}[h]{\jmcm}
          \hspace{-5mm} \includegraphics[width=\imcm]{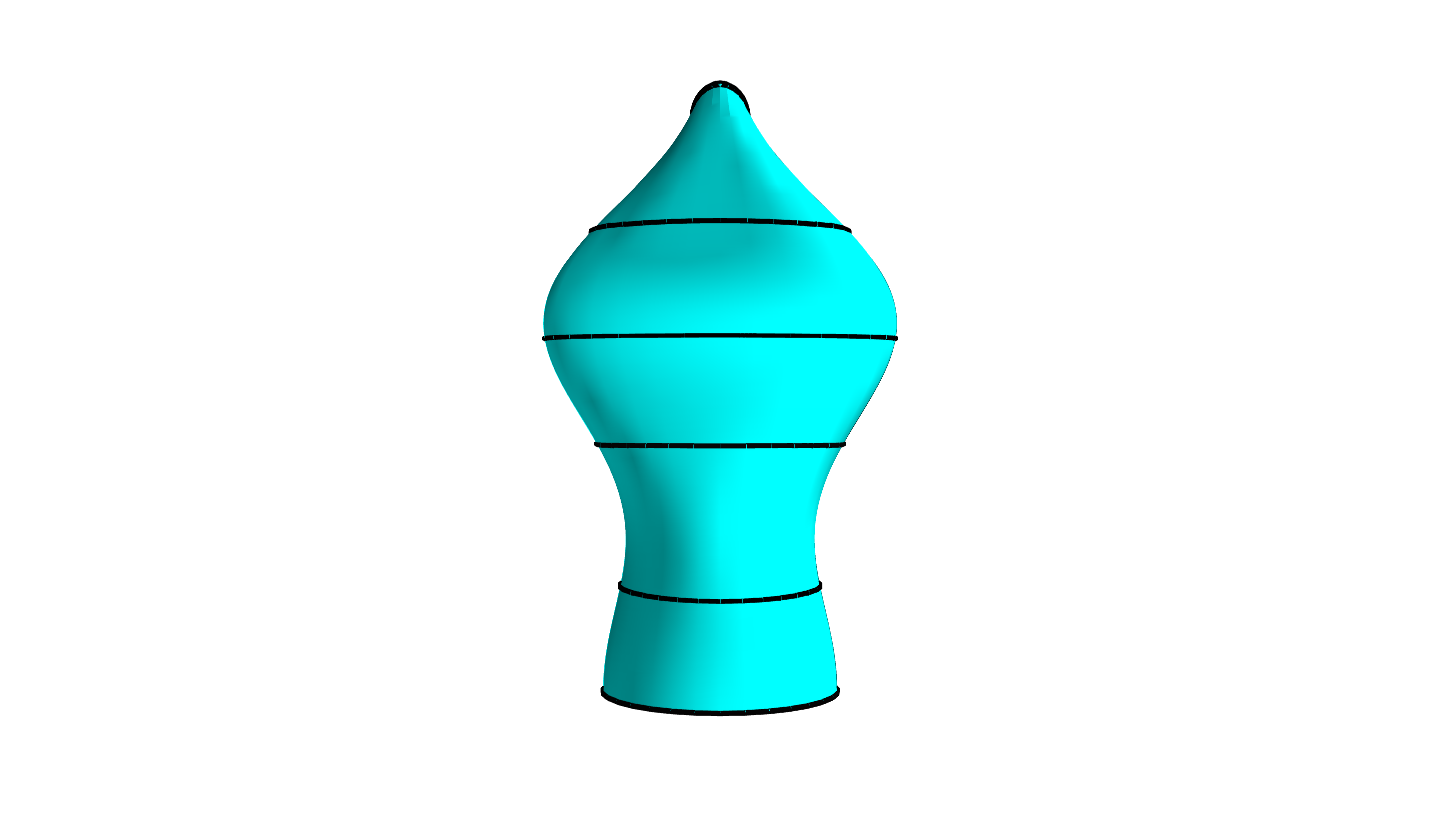}
        \end{minipage}&
        \begin{minipage}[h]{\jmcm}
          \hspace{-10mm} \includegraphics[width=\imcm]{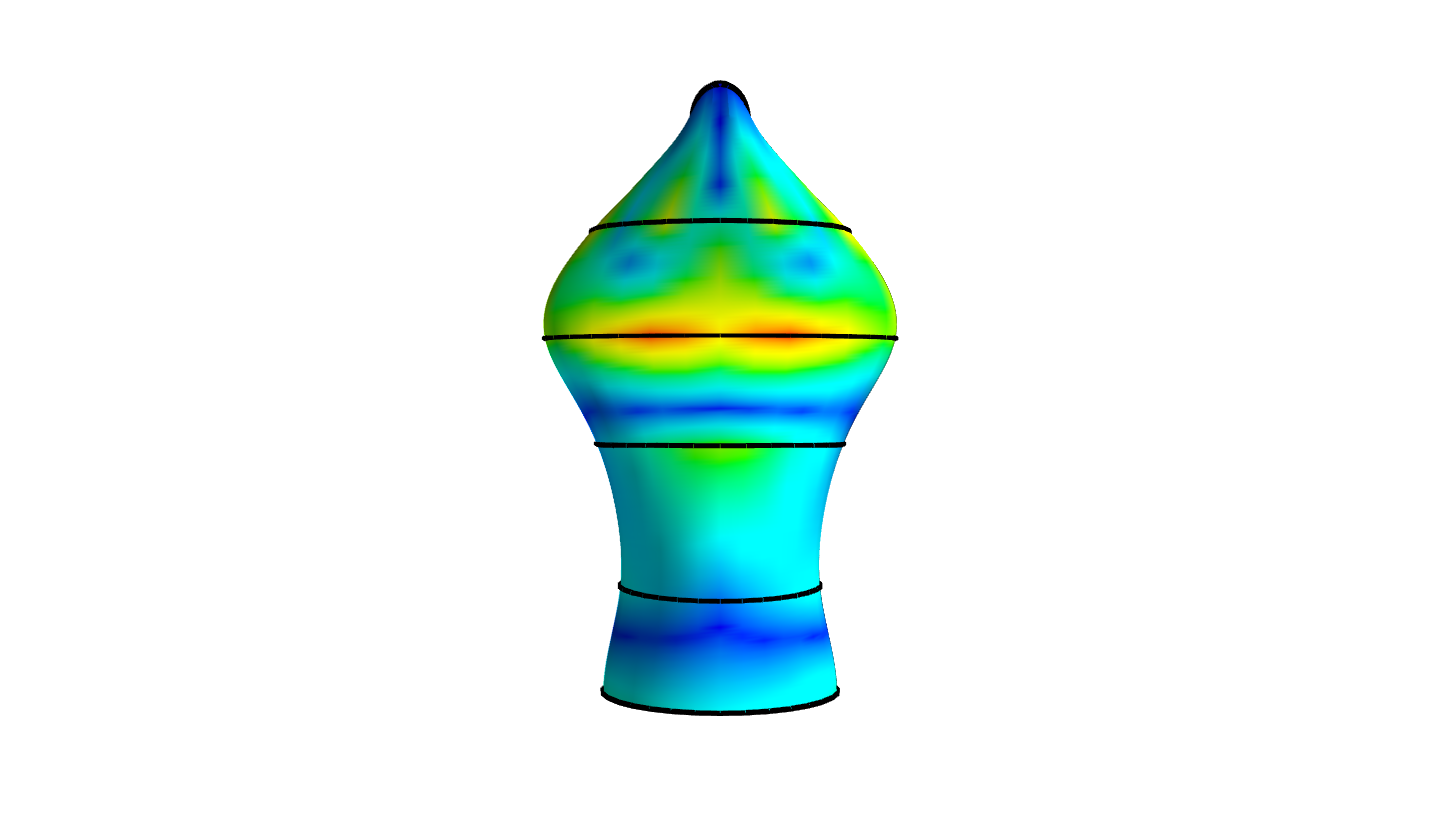}
        \end{minipage}\\
 \hspace{-10mm}       \begin{minipage}[h]{\jmcm}
          \includegraphics[width=\imcm]{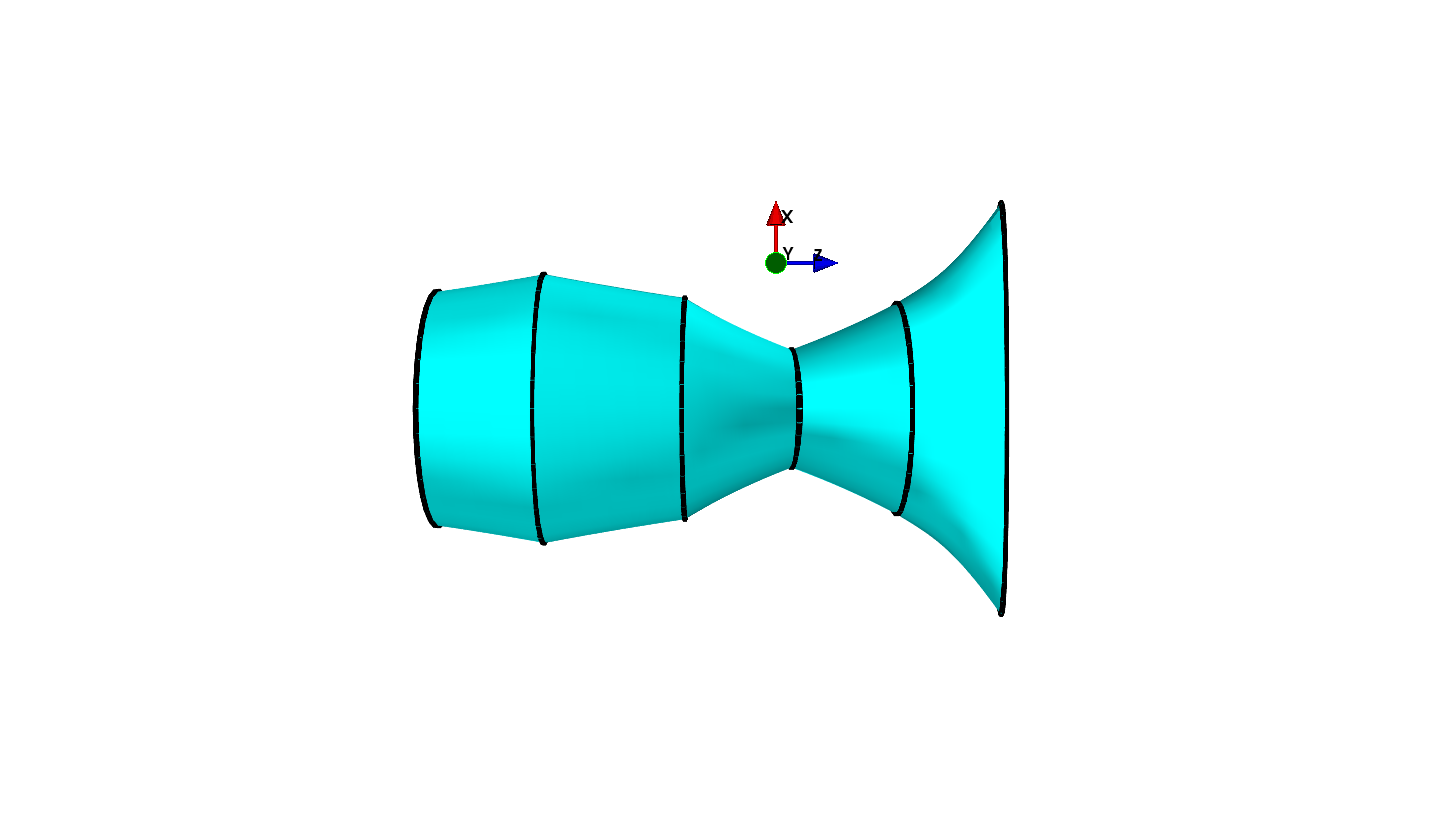}
        \end{minipage}&
        \begin{minipage}[h]{\jmcm}
         \hspace{-5mm} \includegraphics[width=\imcm]{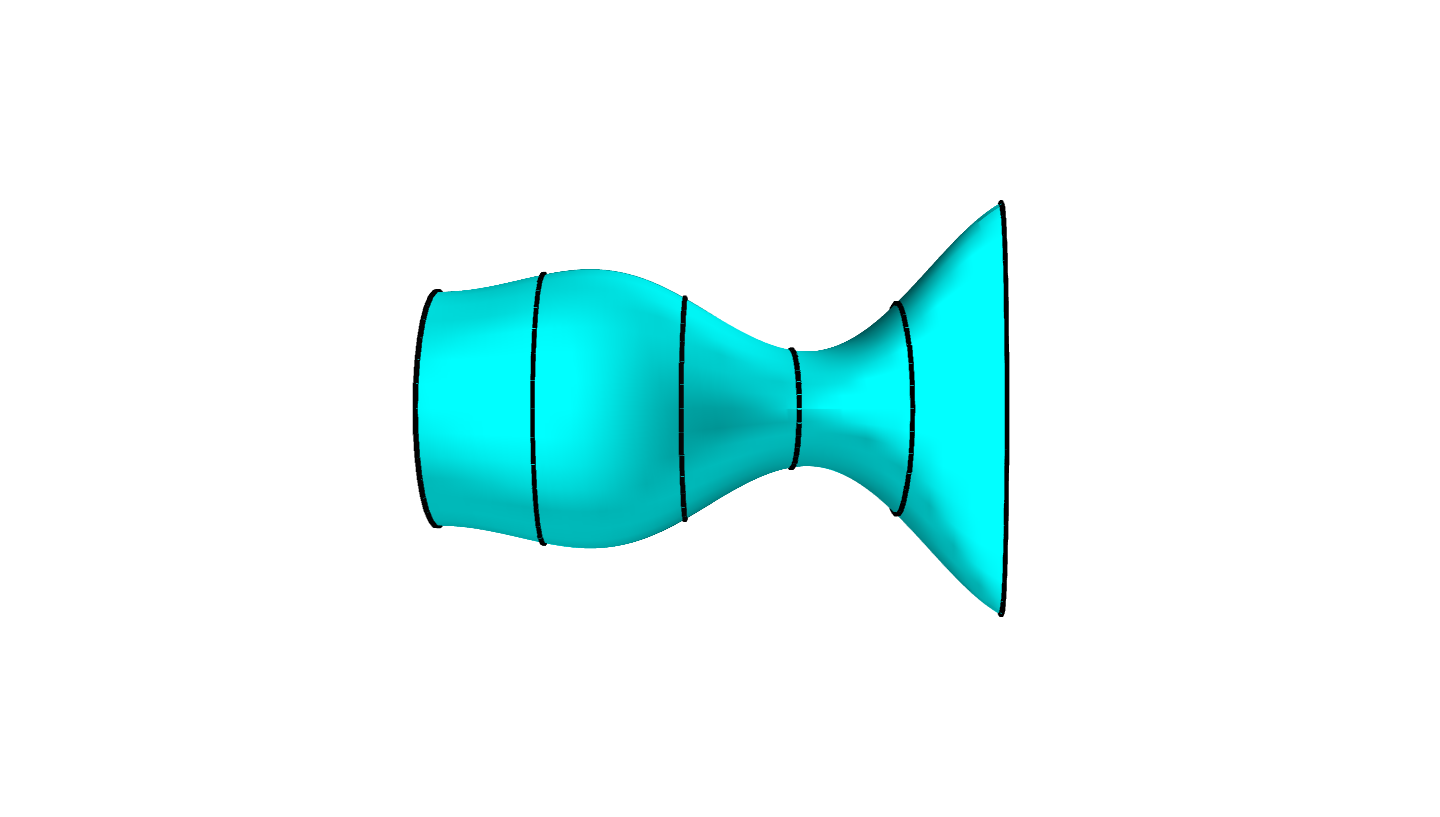}
        \end{minipage}&
        \begin{minipage}[h]{\jmcm}
          \hspace{-10mm} \includegraphics[width=\imcm]{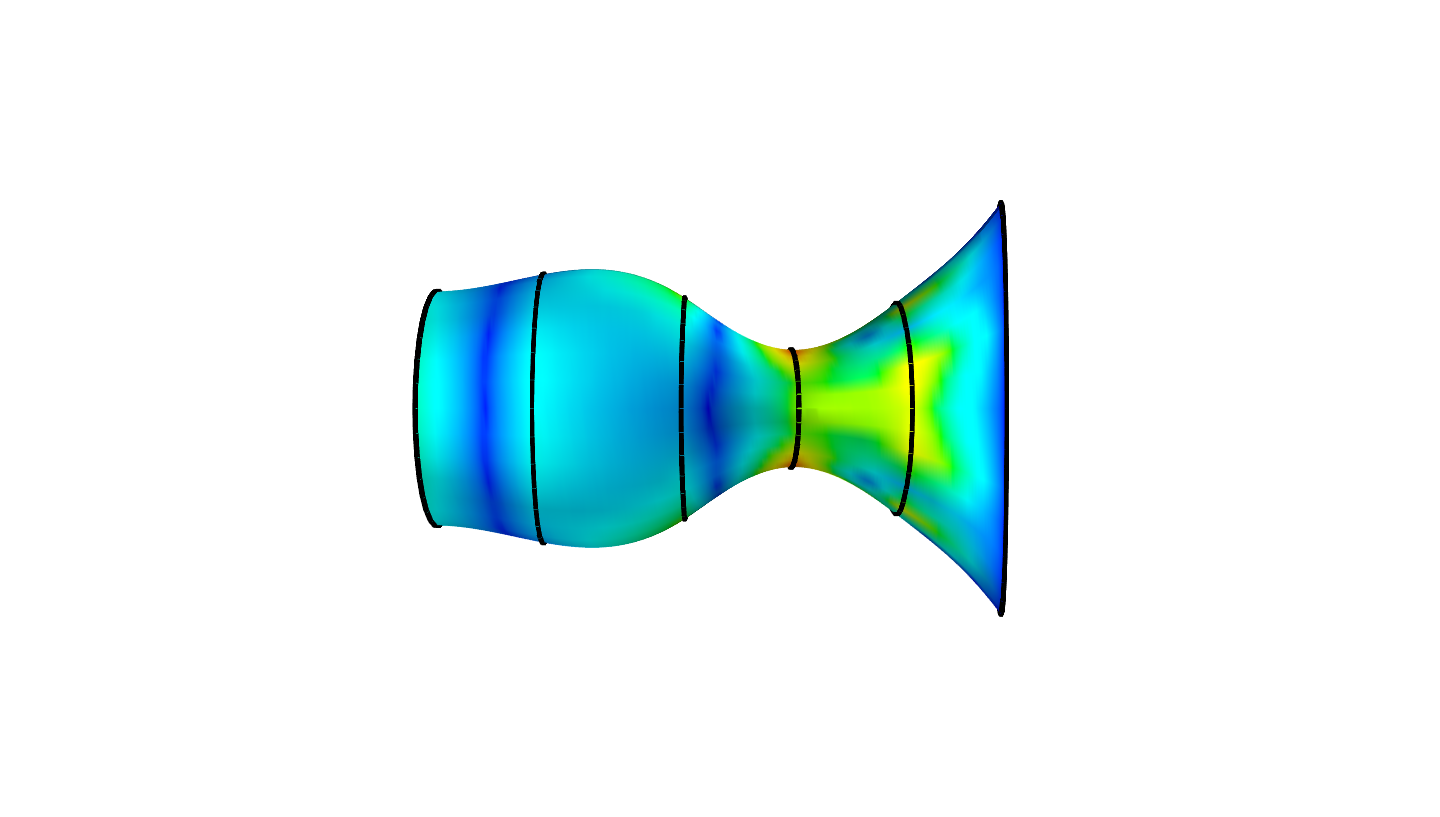}
        \end{minipage}
      \end{tabular}

  \caption{Comparison with piecewise geodesic evolution. On this example the middle column corresponds to 3 different views of a synthetic shape evolution (here the time axis is vertical in the first two rows and horizontal in the last row). The first column correspond to the piecewise geodesic evolution one can get from (\ref{eq:21.12.1}) and the last one to the shape spline estimation.}
  \label{fig:10.2.0}
\end{figure}
We display an experiment on a synthetic evolution where $t\to x^D_t$ is 
given by a simple analytic formula where a circle evolves smoothly into 
an ellipse by increasing its eccentricity through time combined with 
a rotation of the principal axis (see three orthogonal views of the synthetic 
object in the middle column of Fig. \ref{fig:10.2.0}). We display the estimated 
evolution in the piecewise geodesic setting given by (\ref{eq:14.2.1}) and  
with a shape spline in Fig \ref{fig:10.2.0}. As expected, the piecewise geodesic 
estimation provides good results but with a loss of regularity at the observation 
points as in the simpler situation of piecewise linear interpolation
in signal processing. The shape spline seems to perform better at the
observation points but also to provide a better estimation \emph{between}
observation points.
\begin{figure}[h]
  \centering
  \begin{tabular}[h]{cc}
\hspace{-10mm}  \includegraphics[width=7.5cm]{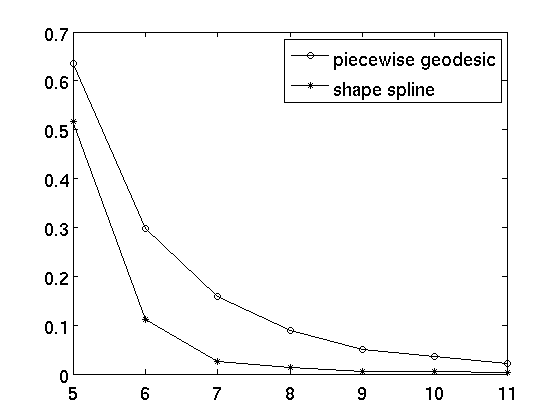} &
\hspace{-10mm}\includegraphics[width=7.5cm]{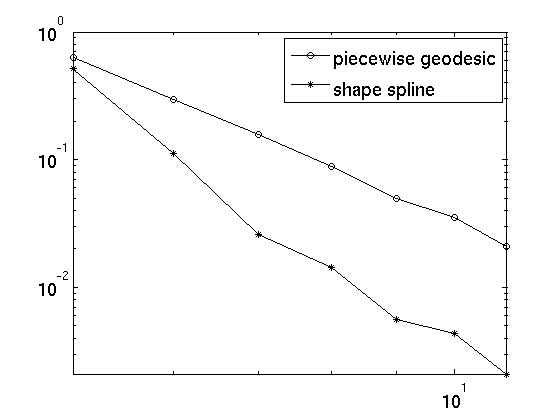}
\end{tabular}
  \caption{Comparison of the $L^2$ error $E$ (see (\ref{eq:14.2.2}))
between piecewise geodesic and shape spline
    estimation. The horizontal axis is the number $M$ of observation
    points and the vertical axis the $L^2$ distance in arbitrary units
    through time
    between the estimated shape and the actual shape provided by 
the synthetic evolution. Right panel, log-log plot.}
\label{fig:14.1.1}
\end{figure}
In Fig. \ref{fig:14.1.1}, we provide a more quantified comparison of
the approximation quality between the estimation process by computing 
the $L^2$ error
\begin{equation}
  \label{eq:14.2.2}
 E\doteq \left(\int_{t_1}^{t_M}|x^D_t-x_t|^2dt\right)^{1/2}
\end{equation}
as a function of the number $M$ of observations. It is quite clear that
the convergence is faster with shape spline interpolation than
with the piecewise geodesic interpolation (we do not go beyond $11$
observation points since the error is small enough to be approaching other
numerical errors in the optimization scheme). The loglog-plot in
Fig. \ref{fig:14.1.1} seems to indicate a polynomial convergence
in $C/M^\alpha$ very similar to the classical situation in
interpolation theory that ($\alpha=2$ for linear spline and 
$\alpha=4$ for cubic spline)\footnote{Note that we have implemented
  here a least square approximation algorithm (see (\ref{eq:25.2}) and
  (\ref{eq:14.2.1}) and not an exact
  interpolation algorithm but with a value $\gamma$ weighting the data
  attachment term high enough to make the data error negligible.} 
\section{A stochastic shape spline model} \label{StoModel}

In this section, we study further the first candidate for the second order model \eqref{eq:24.2} of stochastic evolutions of shapes. We prove that the solutions are well defined for all times and present some simulations that highlight some features of this model.

\subsection{Non blow-up result}\label{nonblowup}

This section is devoted to study the well-posedness of the SDE \eqref{eq:24.2} introduced as our generative growth model in subsection \ref{stochmodel}. The random force is chosen to be an increment of the Brownian motion though we could also have introduced a Levy process in the evolution of the momentum to account for sudden activations of cells. It would turn our growth model into a more realistic one. However, this Brownian perturbation is the first step toward such a model and we will now discuss its feasibility from the mathematical point of view. We will prove that the solutions of the SDE do not blow up in finite time a.s.

The stochastic differential system is:
\begin{subequations} \label{M1}
\begin{align}
& dp_t = -\partial_x H_0(p_t,x_t) \, dt + \ve dB_t \, ,   \\ 
& dx_t = \partial_p H_0(p_t,x_t) \, dt\,. 
\end{align}
\end{subequations}
\noindent
Here, $\ve$ is a constant parameter and $B_t$ is a Brownian motion on $\R^{dn}$. We will work with the Gaussian kernel but this can be directly extended to other kernels. 
When $\ve$ is constant, there is no difference studying these stochastic differential equations with the Ito integral or Stratonovich one. However, for a general variance term, we will use the Ito stochastic integral. From the theorem of existence and uniqueness of solution of stochastic differential equation under the linear growth conditions, we can work on the solutions of such equations for a large range of kernels. 
Yet in our case the Hamiltonian is quadratic, and the classical results for existence and uniqueness of stochastic differential equations only prove that the solution is locally defined. In the deterministic case, this quadratic property could imply existence of a blow-up. To prove that the solutions do not blow up in finite time in the deterministic case ($\ve = 0$), we can use the fact that the Hamiltonian of the system is constant in time.
By adapting that proof (also closely related to the proposition 
\ref{Thm:23.2.1}) and controlling the Hamiltonian, we will prove that the solutions are defined for all time. 
\\
A first remark we will use is the following, for any $\alpha \in \ms{R}^d$ and $z \in \ms{R}^d$,
\begin{equation} \label{IneqRKHS}
\langle \alpha,K_V(z,z)\alpha \rangle_{\mathbb{R}^d} \leq C^2 |\alpha|^2_{\mathbb{R}^d} \,.
\end{equation}

Thus, we introduce the stopping times defined as follows: let $M>0$ be a constant and
\begin{equation} \label{stoppingtimes}
\tau_M = \{ t \geq 0 \, | \,\max(|x_t|,|p_t|) \geq M \} \, , 
\end{equation}
let also $\tau_\infty=\lim_{M\to\infty}\uparrow
\tau_M$ be the explosion time. 
Differentiating $H_0(p_{t \wedge \tau_M},x_{t \wedge \tau_M})$ with respect to $t$, we get on $(t<\tau_M)$:
\begin{equation*}
dH_0(t) = \partial_x H_0(p_t,x_t)dx_t + \partial_p H_0(p_t,x_t)dp_t  +  \sum_{i=1}^n \text{tr}(K_V(x_i(t),x_i(t)))\frac{\ve^2}{2}dt \,.
\end{equation*}
In the deterministic case the Hamiltonian is constant, whereas here the stochastic perturbation gives
$$ \partial_x H_0(p_t,x_t)dx_t + \partial_p H_0(p_t,x_t)dp_t = \ve \partial_p H_0(p_t,x_t) dB_t \,.$$
Thus
$$
\int_0^{T \wedge \tau_M} dH_0(t) = \int_0^{T \wedge \tau_M} \ve \langle \partial_p H_0(p_t,x_t), dB_t \rangle  + \int_0^{T \wedge \tau_M} \sum_{i=1}^n \text{tr}(K_V(x_i(t),x_i(t)))\frac{\ve^2}{2}dt \, ,  \nonumber$$
and
\begin{equation}
  \label{eq:9.3.1}
  E[H_0(p_{T \wedge \tau_M},x_{T \wedge \tau_M})] \leq H_0(0) + E(C^2\frac{\ve^2}{2} dn \, T \wedge \tau_M) \leq H_0(0) + C^2\ve^2 dn T \,. 
\end{equation}
Now, we aim at controlling $x_{t \wedge \tau_M}$ using the control on $dx_t$ given by $|\partial_p H_0(p_t,x_t)|_{\infty} \leq C \sqrt{H_0(p_t,x_t)}$:
  \begin{multline}
    |x_{\tau_M\wedge t}|\leq |x_0|+\int_0^{\tau_M\wedge
      t}CH_0(p_s,x_s)^{1/2}ds
\leq |x_0|+\int_0^{\tau_M\wedge
      t}CH_0(p_{s\wedge \tau_M},x_{s\wedge\tau_M})^{1/2}ds\\
\leq  A_t\doteq  |x_0|+\int_0^{\tau_\infty\wedge
      t}CH_0(p_{s\wedge \tau_\infty},x_{s\wedge\tau_\infty})^{1/2}ds\,.
  \end{multline}
However, $0\leq A_t$ $P$ a.s.  and 
by monotone convergence theorem (recall that $H_0$ is non negative),
$$E(A_t)=\lim_{M\to\infty}(|x_0|+E\left(\int_0^{t\wedge\tau_M} CH_0(p_{s\wedge
  \tau_M},x_{s\wedge \tau_M})^{1/2}ds \right).$$
Also,
 \begin{multline}
E\left(\int_0^{t\wedge\tau_M} H_0(p_{s\wedge
  \tau_M},x_{s\wedge \tau_M})^{1/2}ds\right)\leq E\left(\int_0^{t} H_0(p_{s\wedge
  \tau_M},x_{s\wedge \tau_M})^{1/2}ds\right)\\
\stackrel{Fub.}{=}\int_0^{t} E\left(H_0(p_{s\wedge
  \tau_M},x_{s\wedge \tau_M})^{1/2}\right)ds\stackrel{Jen.}{\leq}\int_0^{t} E\left(H_0(p_{s\wedge
  \tau_M},x_{s\wedge \tau_M})\right)^{1/2}ds\\
\stackrel{CS+~(\ref{eq:9.3.1})}{\leq}\sqrt{t}\left(\int_0^t(H_0(0)+C^2\ve^2 dns)\,
  ds\right)^{1/2} \, . \label{eq:1'}
\end{multline}  
We deduce 
$$E(A_t)\leq |x_0|+C\sqrt{t}\left(\int_0^t(H_0(0)+C^2\ve^2 dns)\,
  ds\right)^{1/2}<\infty\text{ and }A_t<\infty\ P\ a.s.$$
and as a consequence $$\limsup_{M\to\infty} |x_{t\wedge\tau_M}| < +\infty \, P \, a.s.$$
We also control the evolution equation of the momentum as follows,
\begin{equation} \label{unederniere}
 |p_{t \wedge \tau_M}| \leq \int_0^{t \wedge \tau_M} |\partial_x H_0(p_s,x_s)| \, ds + | p_0 + \,\int_0^{t \wedge \tau_M} \ve dB_s \, | \, .
\end{equation}
Now we use the assumption \eqref{C1inj} to control $\partial_x H_0(p,x)$:
\begin{equation*}
|\partial_x H_0(p,x)| \leq |p| |dv(x)| \leq C|p| H_0^{1/2} \, .
\end{equation*}
We rewrite inequality \eqref{unederniere} and we use Gronwall's Lemma to get:
\begin{align*}
& |p_{t \wedge \tau_M}| \leq \int_0^{t \wedge \tau_M}  C|p_s| \, H_0(p_{s},x_{s})^{1/2} \, ds + | p_0 + \,\int_0^{t \wedge \tau_M} \ve dB_s \, | \, , \\
& |p_{t \wedge \tau_M}| \leq \left(|p_0| +  \sup_{u \leq t} | \,\int_0^{u \wedge \tau_M} \ve dB_s \, | \right) e^{\int_0^{t \wedge \tau_M} C H_0(p_s,x_s)^{1/2} ds } \, , \\
& |p_{t \wedge \tau_M}| \leq \left(|p_0| +  \sup_{u \leq t\wedge \tau_{\infty}} | \,\int_0^{u} \ve dB_s \, | \right) e^{\int_0^{t \wedge \tau_{\infty}} C H_0(p_s,x_s)^{1/2} ds } \, .
\end{align*}
The first term on the right-hand side $|p_0| +  \sup_{u \leq t\wedge \tau_{\infty}} | \int_0^{u} \ve dB_s \, |$ is bounded by $ |p_0| +  \sup_{u \leq t} | \,\int_0^{u} \ve dB_s \, | < \infty \, P \, a.s.$ and with inequality \eqref{eq:1'} we have that $$  e^{\int_0^{t \wedge \tau_{\infty}} C H_0(p_s,x_s)^{1/2} ds} < \infty \, P \, a.s. $$
Since on $(\tau_\infty\leq t)$ one
has $$\lim_{M\to\infty}\max(|x_{t\wedge\tau_M}|,|p_{t\wedge
  \tau_M}|)= \lim_{M \to \infty} |p_t| = \infty \, ,$$ we deduce $P(\tau_\infty\leq t)=0$ and
$\tau_\infty=+\infty$ almost surely.

\vspace{0.3cm}

\noindent
We have proved for the case $\ve(p,x) = \ve Id$,

\begin{Thm} \label{landmarkcase}
Under assumption \eqref{C1inj}, the solutions of the stochastic differential equation defined by
\begin{eqnarray*} \label{sytem_sto_land}
 dp_t &=& -\partial_x H_0(p_t,x_t)dt + \ve(p_t,x_t) dB_t   \\ 
 dx_t &=& \partial_p H_0(p_t,x_t)dt. 
\end{eqnarray*}
are non exploding when $\ve: \ms{R}^{nd} \times \ms{R}^{nd} \mapsto L(\ms{R}^{nd})$ is a \lip and bounded map.
\end{Thm}

\begin{proof}
To extend the proof to the case when $\ve$ is a \lip and bounded map of $p$ and $x$, we can prove that the preceding inequalities are still valid.

First, with the \lip property of $\ve$ the solutions are still defined locally. The Ito formula is now written as, on $(t < \tau_M)$
\begin{equation*}
dH_0(t) = \partial_x H_0(p_t,x_t)dx_t + \partial_p H_0(p_t,x_t)dp_t + \frac{1}{2}\text{tr}(\ve^T(p_t,x_t)K_{x_t}\ve(p_t,x_t))dt \,.
\end{equation*}
where $K_{x}$ is block matrix defined by $K_{x}\doteq (K_V(x_i,x_j))_{1\leq i,j\leq n}$.
\noindent

We still have the inequality (\ref{eq:9.3.1}) with 
$$\text{tr}(\ve^T(p_t,x_t)K_{x_t} \ve(p_t,x_t)) \leq (Cnd|\ve|_{\infty})^2$$ 
if $|\ve(p,x) w|^2 \leq |\ve|_{\infty}^2 |w|_{\infty}^2$ where $|\epsilon|_\infty$ denotes the supremum norm. Indeed, if $(e_i)_{i\in [1,nd]}$ the canonical basis of $\ms{R}^{nd}$, denoting $\ve\doteq\ve(x,p)$, we have 
\begin{equation*}
\text{tr}(\ve^t K_x \ve)  = \sum_{i=1}^{nd} \langle \ve(e_i),K_x\ve(e_i) \rangle  
\leq \lambda^{*}_{K_x} \sum_{i=1}^{nd} \langle \ve(e_i), \ve(e_i) \rangle\,, 
\end{equation*}
where $\lambda^{*}_{K_x}$ is the largest eigenvalue of $K_x$. We have $\lambda^{*}_{K_x}\leq \text{tr}(K_x)=\sum_{i=1}^n 
\text{tr}(K_V(x_i,x_i))$ and using (\ref{IneqRKHS}) we get $\text{tr}(K_V(x_i,x_i))\leq d\lambda^*_{K_V(x_i,x_i)}\leq dC^2$ so that $\lambda^{*}_{K_x}\leq C^2nd$. Hence,
$$
\text{tr}(\ve^t K_x \ve)  \leq C^2 nd \sum_{i=1}^{nd} |\ve|_{\infty}^2 \leq (Cnd|\ve|_{\infty})^2 \,.
$$ 
Thus we get,
\begin{align*}
&\int_0^{T \wedge \tau_M} dH_0(t) \leq \int_0^{T \wedge \tau_M} \langle \partial_p H_0(p_t,x_t), \ve(p_t,x_t)dB_t \rangle + \int_0^{T \wedge \tau_M} \frac{(Cnd|\ve|_{\infty})^2}{2}dt \, , \\
& E[H_0(T \wedge \tau_M)] \leq H_0(0) + E(\frac{(Cnd|\ve|_{\infty})^2}{2} \, T \wedge \tau_M) \leq H_0(0) + (Cnd|\ve|_{\infty})^2\,T \,,
\end{align*}
and all the remaining inequalities follow easily thanks to the control on $H_0$ and the bound on $\ve$.
\end{proof}

Once this stochastic model is well-posed on landmark space, the question of its extension to shape spaces naturally arises. It can be proved that this stochastic model does have an extension to the infinite dimensional case: in the case of $L^2(\ms{R}/\ms{Z},\ms{R}^2)$, the natural extension of the Brownian motion on the landmark space is a cylindrical Brownian motion on $L^2(\ms{R}/\ms{Z},\ms{R}^2)$. It may be somewhat surprising to deal with such irregular noise on the momentum variable, we stress the fact that this noise is read by the kernel which strongly regularizes the noise. Though we will not develop it further, it proves that this model has a consistent extension to continuous-shape spaces and could be used to deal with random evolutions of continuous shapes. 
% \\
% \textsl{[Peut-etre mettre ailleurs:]}The next issue that will be addressed in future work is the estimation of a mean trajectory within the spline model with a stochastic term.

\subsection{Simulations}
With these simulations we illustrate the interesting features we observed above. First, this model gives realistic perturbations of geodesics contrary to a first order model such as a Kunita flow. A simulation of a Kunita flow is illustrated in figure Fig.~\ref{Kunita} where the evolution of $40$ points on the unit circle is represented under a Gaussian kernel of width $0.9$. The time evolution has, as expected, the roughness of a Brownian motion and the space variation is smoother due to the kernel. In comparison, our stochastic model gives smoother evolutions in time as in Fig.~\ref{GeodesicNoise1_7} and Fig.~\ref{GeodesicNoise3_5}.

\begin{figure}[htbp]
  \centering	\includegraphics[width=5.5cm]{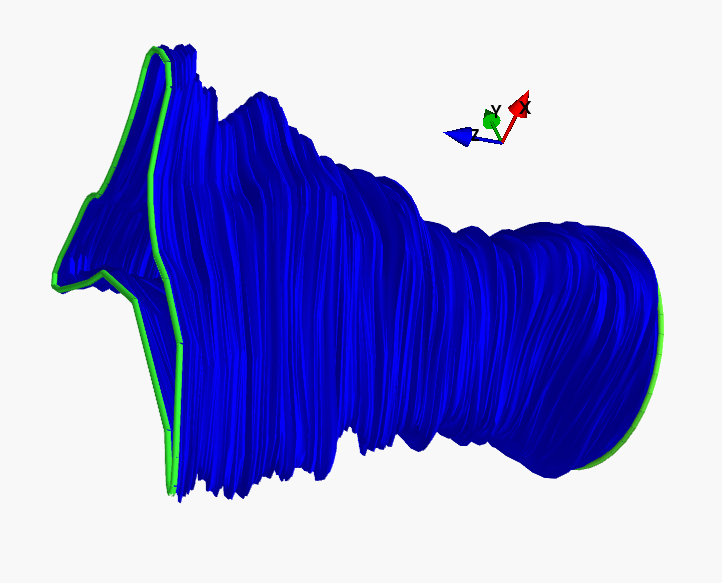}
 \caption{\footnotesize{A simulation of Kunita flow with $40$ points on the unit circle on the left of the figure. The $z$ axis (blue arrow) represents the time.}}
 \label{Kunita}
\end{figure}
Second, our stochastic model is a perturbation of a geodesic evolution and this nice property is illustrated in figures Fig.~\ref{GeodesicNoise0_9}~-~\ref{GeodesicNoise3_5}. 

Figure Fig.~\ref{GeodesicEvolution} shows the geodesic evolution of $40$ equidistributed points on the unit circle for a Gaussian kernel of width $1.0$, the target configuration for the landmarks is obtained through a simple affine transformation that gives the final ellipse. On these simulations the color change only represents time. 
Figures Fig.~\ref{GeodesicNoise0_9}~-~\ref{GeodesicNoise3_5} represent
stochastic perturbations of the previous geodesic; we progressively
increase the standard deviation $\epsilon$ of the noise from
$\sqrt{n}\epsilon=0.9$ to $1.7$ and finally $\sqrt{n}\epsilon=3.5$
(the noise is rescaled w.r.t. the number of landmarks to converge to a
well defined  SPDE at the limit (see \cite{VialardThesis})). Each of
these three figures represents one Monte-Carlo simulation of the 
stochastic model with a simple Euler scheme.
\begin{figure}[htbp]
\hspace{-5mm}
 \begin{minipage}{.5\linewidth}
  \centering	\includegraphics[width=6cm]{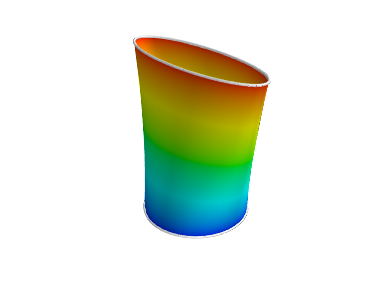}
 \caption{\footnotesize{Geodesic evolution - White unit circle as initial shape.}}
 \label{GeodesicEvolution}
\end{minipage} \hfill
\begin{minipage}{.5\linewidth}
\centering\includegraphics[width=6cm]{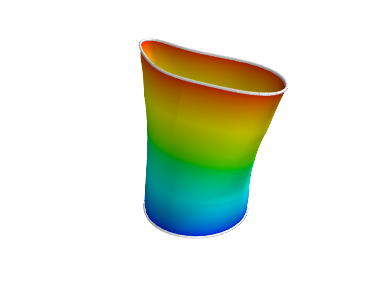}
\caption{\footnotesize{White noise perturbation of the geodesic (same
    initial momentum $p_0$)}, $\sqrt{n}\epsilon=0.9$ }
  \label{GeodesicNoise0_9}
 \end{minipage} \hfill
\end{figure}

\begin{figure}[htbp]
 \hspace{-5mm}\begin{minipage}{.5\linewidth}
  \centering	\includegraphics[width=6cm]{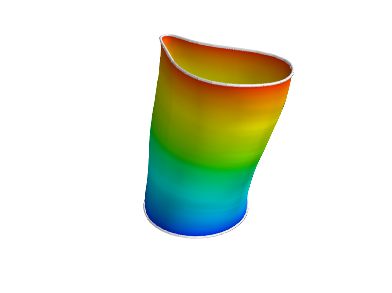}
 \caption{\footnotesize{Increasing the variance of the noise, $\sqrt{n}\epsilon=1.7$}}
 \label{GeodesicNoise1_7}
\end{minipage} \hfill
\begin{minipage}{.5\linewidth}
\centering\includegraphics[width=6cm]{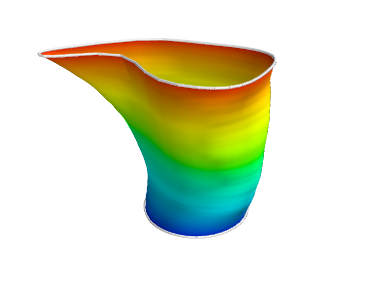}
\caption{\footnotesize{Increasing the variance of the noise, $\sqrt{n}\epsilon=3.5$}}
  \label{GeodesicNoise3_5}
 \end{minipage} \hfill
\end{figure}

The simulations in figures Fig.~\ref{MCnoise+}
show the position at time $1$ of the $40$ points for $5$ Monte-Carlo
simulations. On the two figures we plotted the initial momentum $p_0$
(attached to the $40$ points) associated with the geodesic from the 
initial circle to the target ellipse. As this model was designed to 
produce random shape
evolutions, it can also be used as a generative engine to produce
random shapes.  Increasing the noise also increases the expectation of
the energy of the system since the Ito formula applied on the
Hamiltonian in subsection \ref{nonblowup} shows a linear growth in
time of $H_0$ proportional to $\ve$. This can be guessed when comparing
the two displayed cases in Fig.~\ref{MCnoise+} since in
the second one the noise is $4$ times bigger. For any statistical
estimation of the model parameters, this property should be somehow
taken into account.

\begin{figure}[htbp]
\hspace{-5mm}
\begin{minipage}[htbp]{1.0\linewidth}
  \begin{minipage}{.5\linewidth}
    \hspace{-11mm}\includegraphics[width=9cm]{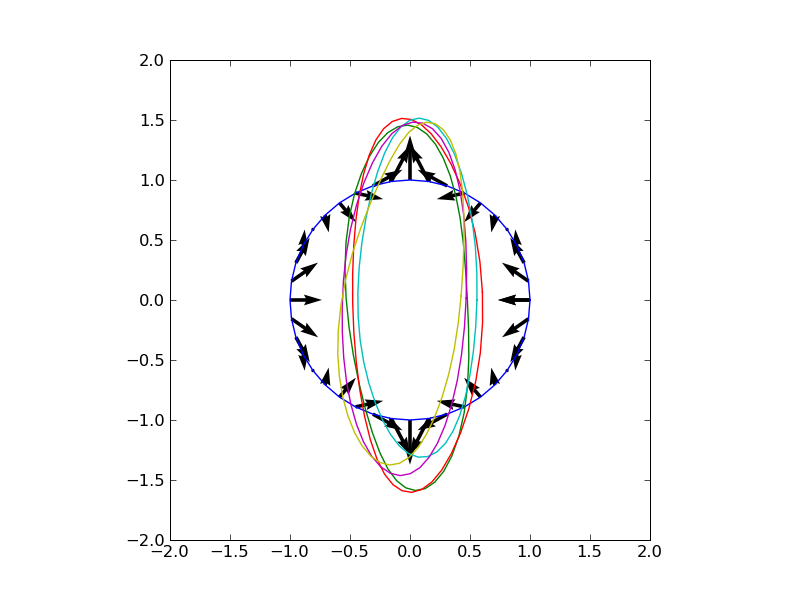}
    % \caption{\footnotesize{}}
    % \label{MCnoise}
  \end{minipage}
  \begin{minipage}{.5\linewidth}
    \hspace{-5mm}\includegraphics[width=9cm]{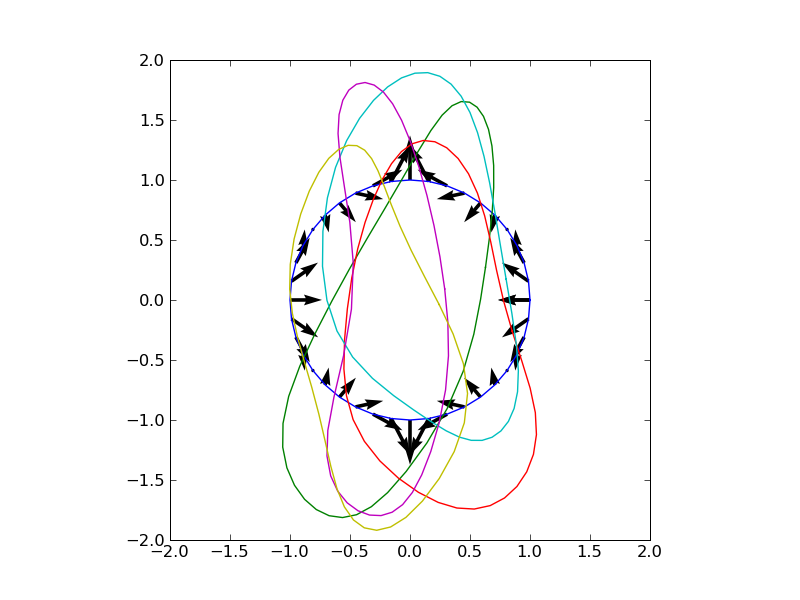}

  \end{minipage} \hfill
\end{minipage}
\caption{\footnotesize{Effects of the standard deviation $\epsilon$ of
    the noise~: $5$ simulations of random deformations of
         the unit circle. The initial momentum $p_0$ is fixed, the
         kernel width $\lambda=1$ (see (\ref{eq:15-4})). Left-hand side~:
         $\sqrt{n}\epsilon=0.25$ ; right-hand side~: $\sqrt{n}\epsilon=1$.}}
  \label{MCnoise+}
\end{figure}
An important feature of this model is that the noise is "read" by 
the kernel. We show in Fig.~\ref{kernel+} simulations of 
the model for a null initial momentum on the same initial shape and we decrease the width of the Gaussian 
kernel from $3$ to $0.3$. The standard deviation of the noise is
constant set to $\sqrt{n}\epsilon=1.0$.
\begin{figure}[htbp]
\hspace{-5mm}
\begin{minipage}[htbp]{1.0\linewidth}
  \begin{minipage}{.5\linewidth}
    \hspace{-11mm}
    \includegraphics[width=9cm]{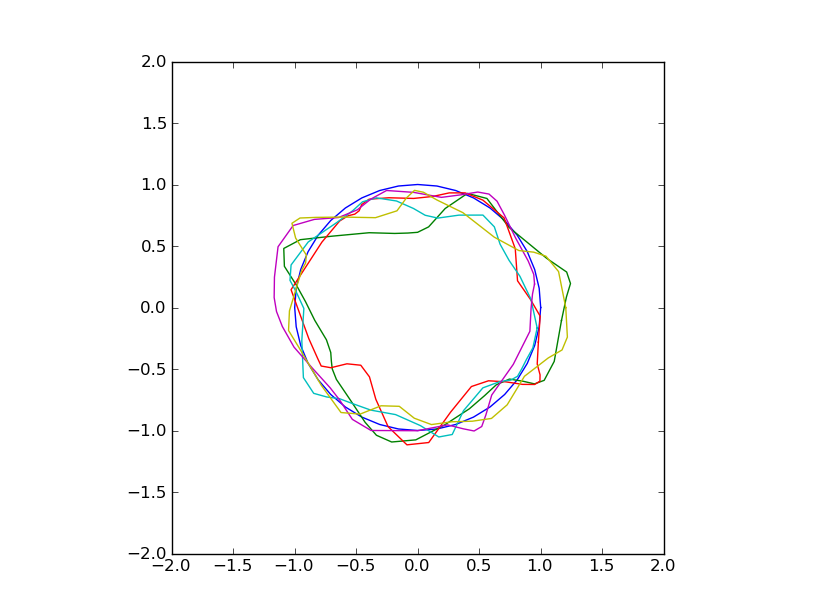}
  \end{minipage}
  \begin{minipage}{.5\linewidth}
    \hspace{-5mm}\includegraphics[width=9cm]{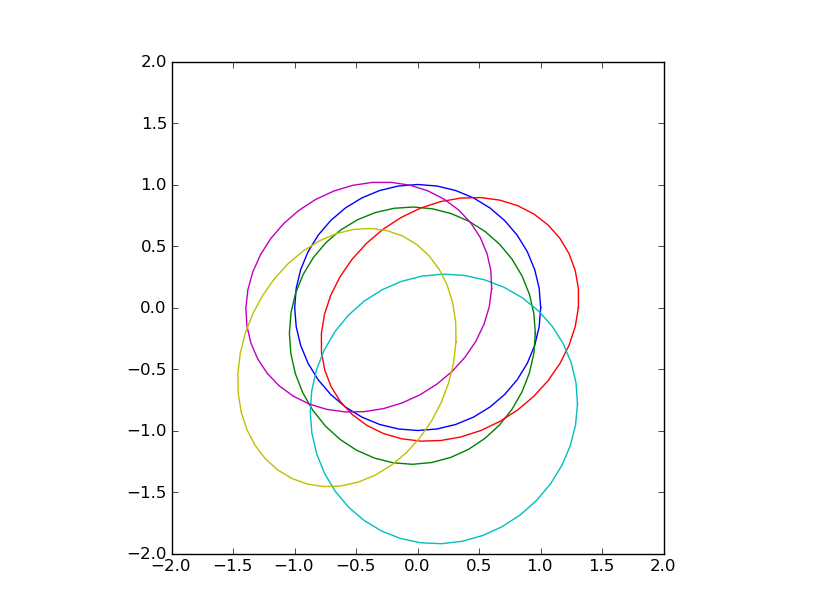}
  \end{minipage} \hfill
\end{minipage}
    \caption{\footnotesize{Effects of the width of the Gaussian
        kernel. Left-hand side~:$\lambda =
        0.3$ ; right-hand side~: $\lambda =
        3.0$.}}
    \label{kernel+}
\end{figure}

These last simulations show the importance of the choice of the kernel and as a by-product the choice of the operator $\ve$ in front of the noise will be also important. Now we can formulate a stochastic model for evolutions of shape that would be closer to realistic evolutions:
\begin{equation} \label{generalized_stochastic_system}
\begin{cases}
dp_t = -\partial_q H_0(p_t,q_t) + u_t + \ve d B_t  \\
dq_t = \partial_p H_0(p_t,q_t) \, ,
\end{cases}
\end{equation}
where $u_t$ is of bounded variations. At this point, we underline that the model parameterization is completely open and it should be tightly related to consistent statistical estimations.
\section{Shape splines on  homogeneous space}
\label{sec:homogeneous}
In this section we provide a more formal and geometrical
picture of what could be an extension of the shape spline  to the
previously mentioned important cases. For this, we need to
introduce some of the standard vocabulary of geometrical
mechanics as developed in \cite{Marsden1999}. We will try  
as much as possible to avoid the conceptual burden of the
intrinsic differential and symplectic geometry through the 
extensive use of local coordinates. Readers looking for
a more intrinsic formulation could refer to  \cite{Marsden1999}.

\subsection{Geometrical setting}
The proposed framework for shape spline is given by three ingredients:
a group $G$ of transformations, a Riemannian manifold $Q$ and a
\emph{left} action of $G$ on $Q$ denoted $(g,q)\to L_q(g)\doteq g\cdot
q$. We will assume also that for any $q\in Q$, $g\to g\cdot q$ is a
surjective submersion (i.e. the differential of $L_q$ has full rank everywhere).
\subsubsection{Local coordinates}
Basically $G$ will be a (finite dimensional) Lie group with Lie algebra
$\mathfrak{G}$ on which we consider a right invariant metric given by
a dot product on $\mathfrak{G}$. 
The infinite dimensional setting where $G$ is a group of
diffeomorphisms is more involved and requires more analytical work 
as in \cite{trouve05:_local_geomet_defor_templ}. 
This is clearly the target setting we have in mind but we want
in this paper to stay away from any complicated analytical developments. To
keep the focus on the global picture, we will assume 
implicitly that we work in a finite dimensional setting 
for which the existence of all the introduced objects is 
straightforward. 

Denoting $q=(q^1,\cdots,q^n)$ local coordinates on $Q$ and since
$(dq^1,\cdots,dq^n)$ is a basis of $T_q^*Q$, we can write any
$\alpha\in T^*_qQ$ as $\alpha=\sum p_idq^i$ so that
$(q^1,\cdots,q^n,p_1,\cdots,p_n)$ are local coordinates on the
cotangent bundle $T^*Q$. Given $q=(q^1,\cdots,q^n)$, we will denote
$p=(p_1,\cdots,p_n)$ a generic element of $T_q^*Q$ and 
$m=(q,p)$ a generic element 
of $T^*Q$ as we did previously in the flat case of 
landmarks.

\subsubsection{Infinitesimal actions and cotangent lift}
The first thing we need is to extend the action of $G$ on $Q$ to an
action of $G$ on the cotangent bundle $T^*Q$. 
Note that the differentiation in $g$ of $g\to q\cdot q$ at $g=\text{Id}_G$ yields
 an infinitesimal action $(\xi,q)\to \xi\cdot q$ for 
$\xi\in \mathfrak{G}$. Differentiation in $q$ yields the action 
$(g,\delta q)\to g\cdot\delta q\in T_{g\cdot q}Q$ for $\delta q\in
T_qQ$ and by duality the action $(g,p)\to g\cdot p\in T_{g\cdot q}^*Q$
for $p\in T_q^*Q$, uniquely defined through the equality 
\begin{equation}
(g\cdot p
 |g\cdot\delta q)\doteq(p|\delta q)\,.\label{eq:18.2.1}
\end{equation}
We denote 
\begin{equation}
(g,\mQ)\to \mQ\cdot g\doteq (g\cdot q,g\cdot p)\label{eq:18.2.6}
\end{equation}
for 
$\mQ=(q,p)\in T^*Q$ the induced action on $T^*Q$. 

In summary, the initial action $g\to g\cdot q$ on $Q$ is naturally
lifted to an action $g\to g\cdot m$ on the cotangent space $T^*Q$
(usually 
called \emph{cotangent lift} \cite{Marsden1999}). 
Differentiating the action $(g,\mQ)\to g\cdot \mQ$ on the
cotangent space $T^*Q$ at $g=\Id_G$, we get the infinitesimal action
on $T^*Q$, which is defined in local coordinates by
$\xi\cdot\mQ=(\xi\cdot q,\xi\cdot p)$ where 
\begin{equation}
(\xi\cdot p |\delta
q)+(p|\xi\cdot\delta q)=0\label{eq:17.2.1}
\end{equation}
as obtained by differentiation of the
conservation equation (\ref{eq:18.2.1}).
\subsection{Euler-Poincar\'e equation}
The initial matching problem between shapes 
in $Q$ is defined as the solution of the optimal control problem with fixed boundary
$$\left|
\begin{array}[h]{l}
  \min_{\xi_t}\frac{1}{2}\int_0^1 (L\xi_t,\xi_t)dt\\
\text{subject to}\\
\dot{q}_t= \xi_t\cdot q_t,\ q_0=q_{\text{init}},\ q_1=q_{\text{targ}}
\end{array}\right.$$
where $L:\mathfrak{G}^*\to\mathfrak{G}$ is the isometry between 
$\mathfrak{G}$ and its dual induced by the metric on $\mathfrak{G}$.

The Hamiltonian associated to the classical matching problem is given in
local coordinates by
$H(q,p,\xi)=(p|\xi \cdot q)-\frac{1}{2}(L\xi|\xi)$
with reduced form
$$H(q,p)\doteq \frac{1}{2}\big(KJ(q,p)\,|\,J(q,p)\big)$$
where $J(q,p)\in\mathfrak{G}^*$ is uniquely defined by
\begin{equation}
(J(q,p)|\xi)=(p|\xi\cdot q)\label{eq:18.2.5}
\end{equation}
for any $\xi\in\mathfrak{G}$ (usually
called the \emph{momentum map}). To obtain the Hamiltonian evolution, we need to compute the variation
of $\delta J$ as a function of the variation $\delta
q$ and $\delta p$ in $q$ and $p$. Introducing in local coordinates the
so-called \emph{symplectic matrix} 
$$\mathbb{J}\doteq\left(
\begin{array}[h]{cc}
0  & \text{Id}_n \\
- \text{Id}_n & 0 
\end{array}\right)$$ 
one checks easily that $(\delta J|
\xi)=(p|\xi\cdot \delta q)+(\delta p |\xi\cdot q)$ so that using
(\ref{eq:17.2.1}) we get 
\begin{equation}
(\delta J|
\xi)= -(\xi\cdot p|\delta q)+(\delta p | \xi\cdot q)= -(\mathbb{J} (\xi\cdot \mQ)|\delta \mQ)
\label{eq:17.2.2}
\end{equation}
and $dH=-\mathbb{J} (KJ\cdot \mQ)$. Since the associated
Hamiltonian evolution equation are given by $\dot{m}=\mathbb{J}dH$ we
get 
\begin{equation}
\dot{m}=Kj\cdot \mQ\text{ with } j=J(q,p)
\label{eq:27.12.1}
\end{equation}
or equivalently
\begin{equation}
\dot{q}=\xi\cdot q\text{ and }
\dot{p}=\xi\cdot p \text{ with } \xi=KJ(q,p)\label{eq:18.2.2}\,.
\end{equation}
This extremely  simple expression of $\dot{m}$ 
in term of the momentum map and the infinitesimal action on the 
cotangent space reveals part of the nice geometrical structure 
underlying the evolution. In this setting, it is interesting to
consider the time evolution of the pair $(q,j)$ instead of the pair
$(q,p)$ since $j$ follows an autonomous equation. Indeed, from
(\ref{eq:17.2.1}) and (\ref{eq:18.2.2}), we get that for 
any $\zeta\in\mathfrak{G}$ we have $(\frac{dj}{dt}\ |\ \zeta)=
(\xi_t\cdot p\ |\ \zeta\cdot q)+(p\ |\ \zeta.(\xi_t\cdot
q))=(p\ |\ \zeta.(\xi_t\cdot q) - \xi_t\cdot (\zeta \cdot q))=
-(p\ |\ \mathrm{ad}_{\xi_t}(\zeta)\cdot q) =
-(j\ |\ \mathrm{ad}_{\xi_t}(\zeta))$ where
$\mathrm{ad}_{\xi}(\zeta)=[\xi,\zeta]$ is the adjoint representation 
of the Lie algebra $\mathfrak{G}$ so that we get 
$$\frac{dj}{dt}+\mathrm{ad}^*_{Kj}j=0\,.$$ 
This equation, called the \emph{Euler-Poincar\'e equation} 
plays a central role in geometric mechanics and more recently in the 
large deformation methods in shape analysis and computational 
anatomy \cite{hty09,Marsden1999}.  

Now consider the perturbed dynamic $\dot{p}=Kj(m).p+u$ with $u\in
T_q^*Q$ or equivalently 
$$\frac{dj}{dt}+\mathrm{ad}^*_{Kj}j=h$$
where $h\doteq J(q,u)$ and the associated optimal control problem for the
state variables $(q,j)\in Q\times \mathfrak{G}^*$ and the cost 
$\frac{1}{2}|u|^2_{q,*}$. The norm
$|u|_{q,*}$ we consider here is the dual norm induced on $T_q^*Q$ by the metric
on $T_qQ$. Let $\mathcal{K}_q:T_q^*Q\to T_qQ$ be the isometry such
that $(u\ |\ \mathcal{K}_qu)=|u|_{q,*}^2$. The Hamiltonian associated to
our new control problem for the costate variable $(p_q,\zeta)\in
T_q^*Q\times\mathfrak{G}$ is given by
\begin{equation}
\mathcal{H}(q,j,p_q,\zeta,u)=% (p_q\ |\ Kj\cdot q)+(-\text{ad}_{Kj}^*j+h\ |\
% p_j)-\frac{1}{2}|u|_{q,*}^2)=
(p_q\ |\ Kj\cdot q)+(-\text{ad}_{Kj}^*j+h\ |\ \zeta)-\frac{1}{2}(\mathcal{K}_qu\ |\ u)\label{eq:18.2.3}
\,.
\end{equation}
To compute its reduced form let us note that from (\ref{eq:17.2.1}),
we get $\frac{\partial}{\partial u}(h\ |\ \zeta)=\zeta\cdot q$. Hence 
$\frac{\partial}{\partial u}\mathcal{H}=0$ implies
$\mathcal{K}_qu=\zeta\cdot q$ giving the reduced Hamiltonian
\begin{eqnarray}
\mathcal{H}(q,j,p_q,\zeta) &= &(p_q\ |\ Kj\cdot q)+(-\text{ad}_{Kj}^*j\ |\
\zeta)+\frac{1}{2}|\zeta\cdot q|^2_q\\
&=&(J(q,p_q)\ |\ Kj)+(j\ |\
\text{ad}_{\zeta}(Kj))+\frac{1}{2}|\zeta\cdot q|^2_q\,,\label{eq:18.2.4}
\end{eqnarray}
where $|\ |_q$ denote the norm on $T_qQ$ given by the Riemmanian
metric on $Q$. The associated Hamiltonian evolution is derived  quite
easily:
\begin{equation}
  \label{eq:19.2.1}\left\{
  \begin{array}[h]{l}
\xi=Kj\\
\frac{dq}{dt}=\xi\cdot q\\
\frac{dp_q}{dt}-\xi.p_q=-\frac{\partial}{\partial
      q}(\frac{1}{2}|\zeta\cdot q|^2_q)\\
\frac{dj}{dt}+\text{ad}_\xi^*(j)=\frac{\partial}{\partial
      \zeta}(\frac{1}{2}|\zeta\cdot q|^2_q)\\
\frac{d\zeta}{dt}+\text{ad}_\zeta(\xi)+K\text{ad}_\zeta^*(j)+KJ(q,p_q)=0
  \end{array}\right.
\end{equation}
Here again, the derivation should be considered at a formal level 
or in a smooth finite dimensional setting since existence results 
are beyond the scope of this paper.
\section{Conclusion}
In this paper, we present new tools for shape evolution analysis
or growth analysis in computational anatomy through the introduction
of second order evolutions. Shape splines seem to overcome some of
the limitations of the previous first order schemes and with their
stochastic counterpart can provide the backbone of new statistical
time regression tools. From various perspectives, shape splines offer 
new interesting mathematical and practical challenges: extensions to
the infinite dimensional case of continuous shapes and to images,
development of efficient and scalable numerical schemes to solve 
the spline estimation problem, integration of time realignment as 
developed in \cite{durlemann}, derivation of more complex stochastic
engines beyond the white-noise situation presented here and development of
consistent statistical schemes in the spirit of \cite{aat07}. 
The solutions to some of these problems appear well within reach.

In this paper we preferred to stick to the finite 
dimensional setting in order to focus on the general picture and avoid more difficult 
analysis, the more technically involved infinite dimensional situation 
has been partially explored in the stochastic case in \cite{VialardThesis}.
\vspace{10mm}
\noindent\textbf{Acknowledgments.}
The authors would like to thank Darryl D. Holm and Colin J. Cotter for a fruitful 
discussion about introducing noise on the momentum variable which was the 
starting point of the stochastic model presented here.
\section{Appendix}
\subsection{Link with cubic splines}
We recall that on a Riemannian manifold $(M,g_M)$, a cubic spline between $(x_0,v_0)$ and $(x_1,v_1)$ is a $C^2$ curve $c:[0,1] \to M$ that minimizes 
\begin{equation} \label{usualsplines}
\mc{I}(c) = \frac{1}{2}\int_0^1 g_M( \nabla_{\dot{c}} \dot{c},\nabla_{\dot{c}} \dot{c} ) \, dt\,.
\end{equation}
The Euler-Lagrange equation for this functional is the following (see \cite{Noakes1,splinesanalyse})
\begin{equation}
\nabla^3_t \dot{c} + R(\nabla_t \dot{c},\dot{c}) \dot{c} = 0\,.
\end{equation}
% Now it can be proven that given boundary conditions, there exists a minimizer of the functional $\mc{I}$ \eqref{usualsplines}.
% \begin{prop} \label{RiemannianSplines}
% If $g$ is defined as $g(u,u) = \langle u^p,k(x,x)u^p \rangle$ then the $g-$cubic splines are the cubic splines for the Riemannian manifold $(\mc{L},k^{-1})$.
% \end{prop}

% \begin{proof}
Let $c: I \to M$ a smooth path, then the covariant derivative is given in coordinates by
$$ \nabla_{\dot{c}} \dot{c} = [\ddot{c}_i + \sum_{k,l} \dot{c}_k \Gamma_{k,l}^i(c) \dot{c}_l]_{i \in [1,r]},$$
and we observe that the second member of the right hand side only depends on $(\dot{c},c)$.
We compare this expression with the Hamiltonian equations (here $M=\mathcal{L}$):
\begin{align}
&\ddot{c} = \frac{d}{dt}k(c,c)p = [\partial_1 k(c,c)](\dot{c}) p + k(c,c) \dot{p} \,, \\
&\ddot{c} = [\partial_1 k(c,c)](k(c,c)^{-1}p)p - k(c,c) \partial_xH + k(c,c)u \,,
\end{align}
where for any $x,y\in \mathbb{R}^{nd}$, $k(x,y)$ denotes the block matrix defined by
$$k(x,y)\doteq (K_V(x_i,y_j))_{1\leq i,j\leq n}\,.$$ 
Now the geodesic equations are given by 
$$ \ddot{c}_i + \sum_{k,l} \dot{c}_k \Gamma_{k,l}^i(c) \dot{c}_l = 0 \; \text{for}\, i \in [1,r]\,,$$
where $\Gamma_{k,l}^i = g_{\mc{L}}( \nabla_{\partial_i}\partial_j,\partial_k )$ are the Christoffel symbols in the chosen coordinates.

However, if $u=0$ we obtain the geodesic equations in the Hamiltonian form.
Then we can identify
$$ -[\sum_{k,l} \dot{c}_k \Gamma_{k,l}^i(c) \dot{c}_l]_{i \in [1,r]} = [\partial_1 k(c,c)](k(c,c)^{-1}p) p - k(c,c) \partial_xH \,,$$
which gives
$$ \nabla_{\dot{c}} \dot{c} = k(c,c)u\,.$$
Now we have
$$ g_{\mc{L}}( \nabla_{\dot{c}} \dot{c}, \nabla_{\dot{c}} \dot{c}) = g_{\mc{L}}(k(c,c)u,k(c,c)u) = \langle u, k(c,c)u \rangle \,,$$
which proves the desired result.
\subsection{Proof of Proposition \ref{prop:10.3.1}}
Let us denote $\partial_{12}K_W(z_1,z_2)\doteq \frac{\partial^2 K_W}{\partial z_1\partial z_2}(z_1,z_2)$ for any $z_1,z_2\in\mathbb{R}^d$ and $\Delta_{z,h,\epsilon}\doteq\frac{1}{\epsilon}(\delta_{z+\epsilon h}-\delta_{z})$ for any $z,h\in\mathbb{R}^d$ and $\epsilon>0$. Let us recall that for any $z_1,z_2\in\mathbb{R}^d$
\begin{equation}
  \label{eq:10.3.1}
  \langle \delta_{z_1},\delta_{z_2}\rangle_{W^*}=K_W(z_1,z_2)\,.
\end{equation}

First, we start with the proof that $\Delta_{z,h,\epsilon}$ converges in $W^*$ when $\epsilon\to 0$. Indeed we get from (\ref{eq:10.3.1}) that
\begin{eqnarray}
\langle \Delta_{z_1,h_1,\epsilon_1},\Delta_{z_2,h_2,\epsilon_2}\rangle_{W^*}&= &\int_0^1\int_0^1 \partial_{12}K_W(z_1+s\epsilon_1 h_1,z_2+t\epsilon_2 h_2)\cdot h_1\otimes h_2\,dsdt\label{eq:10.3.2}\\
& = & \partial_{12}K_W(z_1,z_2)\cdot h_1\otimes h_2 + o((|\epsilon_1h_1|+|\epsilon_2h_2|)|h_1||h_2|)\,,
\end{eqnarray}
so that
$|\Delta_{z,h,\epsilon}-\Delta_{z,h,\epsilon'}|^2_{W^*}=o(|\epsilon|+|\epsilon'|)$. Since
$W^*$ is complete, $\Delta_{z,h,\epsilon}$ converges in $W^*$ to a
limit point denoted $\delta'_{z,h}$ such that $\langle
\delta'_{z_1,h_1},\delta'_{z_2,h_2}\rangle_{W^*}=\lim \langle
\Delta_{z_1,h_1,\epsilon},\Delta_{z_2,h_2,\epsilon}\rangle_{W^2}=\partial_{12}K_W(z_1,z_2)\cdot
h_1\otimes h_2$. In particular $h\to \delta'_{z,h}$ is a continuous
linear mapping from $\mathbb{R}^{nd}$ to $W^*$.

Now, we get from (\ref{eq:10.3.2}) 
$$|\delta_{z,h}-(\delta_z+\delta'_{z,h})|^2_{W^*}=o(|h|^3)$$
so that $\psi$ is differentiable at any location $x\in\mathbb{R}^{nd}$
with differential $\psi'(x)$ given by 
$$\psi'(x)\Delta x\doteq
\frac{1}{n}\sum_{i=1}^n \delta'_{x_i,\Delta x_i}\,.$$
Thus, using again (\ref{eq:10.3.2}), we get
$\langle \psi'(x)\Delta x,\psi'(x)\Delta
x\rangle_{W^*}=\frac{1}{n^2}\sum_{i,j=1}^n\partial_{12}K(x_i,x_j)\cdot
\Delta
x_i\otimes \Delta x_j$ which proves Proposition \ref{prop:10.3.1}. 
\bibliographystyle{abbrv}
\bibliography{SecOrdLand,deforNew,shoot} 
\end{document}